\documentclass[11pt]{amsart}

\usepackage{amsmath, amssymb ,amsthm, amsfonts, amsgen,color}
\usepackage{bbm}
\usepackage{tikz}
\usepackage{color}

\numberwithin{equation}{section}

\newcommand{\R}{{\mathbb{R}}}

\newcommand{\beq}{\begin{equation}}
\newcommand{\eeq}{\end{equation}}

\newcommand{\weaklystar}{\rightharpoonup^\ast}
\newcommand{\weakly}{\rightharpoonup}

\def\dist{\text{dist}}

\def\XXint#1#2#3{{\setbox 0=\hbox{$#1{#2#3}{\int}$}
\vcenter{\hbox{$#2#3$}}\kern-.5\wd0}}

\makeatletter
\def\rightharpoonupfill@{\arrowfill@\relbar\relbar\rightharpoonup}

\newcommand{\xrightharpoonup}[2][]{\ext@arrow
\rightharpoonupfill@{#1}{#2}} \makeatother

\newtheorem{Theorem}{Theorem}[section]
\newtheorem{Lemma}[Theorem]{Lemma}
\newtheorem{Proposition}[Theorem]{Proposition}
\newtheorem{Corollary}[Theorem]{Corollary}

\newtheorem{Remark}[Theorem]{Remark}
\newtheorem{Definition}[Theorem]{Definition}
\newtheorem{Example}[Theorem]{Example}

\title[Loss of double-integral character during relaxation]{Loss of double-integral character during relaxation}

\author{Carolin Kreisbeck}
\address{Mathematisch Instituut, Universiteit Utrecht, Postbus 80010, 3508 TA Utrecht, The Netherlands}
\email{c.kreisbeck@uu.nl}

\author{Elvira Zappale}
\address{D.I.In., Universit\`a degli Studi di Salerno, Via Giovanni Paolo II 132, 84084 Fisciano, SA, Italy}
\email{ezappale@unisa.it }

\begin{document}

\maketitle
\thispagestyle{empty}

 \begin{abstract} 
We provide explicit examples to show that the relaxation of functionals
\begin{align*}
L^p(\Omega;\R^m) \ni u\mapsto \int_\Omega\int_\Omega W(u(x), u(y))\, dx\, dy,
\end{align*}
where $\Omega\subset\R^n$ is an open and bounded set, $1<p<\infty$ and $W:\R^m\times \R^m\to \R$ a suitable integrand, is in general not of double-integral form. 

This proves an up to now open statement in [Pedregal, \textit{Rev.~Mat.~Complut.} \textbf{29} (2016)] and [Bellido \& Mora-Corral, \textit{SIAM J.~Math.~Anal.} \textbf{50} (2018)]. The arguments are inspired by recent results regarding the structure of (approximate) nonlocal inclusions, in particular, their invariance under diagonalization of the constraining set. 
For a complementary viewpoint, we also discuss a class of double-integral functionals for which relaxation is in fact structure preserving and the relaxed integrands arise from separate convexification. 
\vspace{8pt}

 
 
 \vspace{8pt}
 
 \noindent\textsc{Date:} \today 
 \end{abstract}

\section{Introduction}
Let $\Omega\subset \R^n$ be a non-empty, open and bounded set and $1<p<\infty$. Moreover, let  $W:\R^m\times \R^m\to \R$ be a lower semicontinuous function satisfying $p$-growth, i.e.,~$ 
W(\xi, \zeta)\leq C(|\xi|^p + |\zeta|^p +1)$ for all $(\xi, \zeta)\in \R^m\times \R^m$ with a constant $C>0$. For any such $W$, we define a double-integral functional
\begin{align}\label{I_W}
I_W(u) =\int_\Omega\int_\Omega W(u(x), u(y))\, dx\, dy
\end{align}
for $u\in L^p(\Omega;\R^m)$. Without loss of generality (see e.g. \cite{Ped97}), one may assume $W$ to be symmmetric, that is, $W(\xi,\zeta)=W(\zeta,\xi)$, for every $(\xi,\zeta)\in \mathbb R^m\times \mathbb R^m$. 

Nonlocal functionals of this type and their inhomogeneous versions with explicit dependence of $W$ on $x, y\in \Omega$ have recently become of increasing interest in the literature.   Besides their nonlocal character, which gives rise to interesting mathematical questions that require the development of new techniques~\cite{BMC18, BeP06, Mun09, Ped16},  this can also be attributed to their relevance in various modern modeling approaches, e.g.~in image processing~\cite{BrN18, FKR15, GiO08}, in machine learning \cite{EDLL14, SOZ17, SlT19}, in the theory of phase transitions~\cite{DFL18, SaV12}, or in continuum mechanics through the theory of peridynamics~\cite{BMP15, ELP13, KMS18, MeD15, Sil00} and crystal plasticity \cite{MMOZ}. 

Under the additional assumption that $W$ is $p$-coercive, i.e., there are constants $c,C>0$ such that
\begin{align*} 
W(\xi, \zeta)\geq c(|\xi|^p+|\zeta|^p) -C\qquad\text{ for all $(\xi, \zeta)\in \R^m\times \R^m$,}
\end{align*}
the existence of minimizers of $I_W$ is guaranteed by the direct method in the calculus of variations, if $I_W$ is $L^p$-weakly lower semicontinuous, or equivalently, if $W$ is separately convex \cite{BeP06, Mun09, Ped16}. In situations when $W$ fails to have this property, 
minimizers of $I_W$ do in general not exist due to oscillation effects.  
 A common strategy to capture the asymptotic behavior of minimizing sequences of $I_W$ is resorting to a related variational problem, called the relaxed problem, which involves 
the $L^p$-weak lower semicontinuous envelope of $I_W$, i.e., for $u\in L^p(\Omega;\R^m)$,
\begin{align}\label{IWrlx}
I_W^{\rm rlx}(u) = \inf \{\liminf_{j\to \infty} I_W(u_j)\,:\, u_j\weakly u\text{ in $L^p(\Omega;\R^m)$}\}.
\end{align} 
The major challenge in relaxation theory lies in finding alternative representations of $I_W^{\rm rlx}$, ideally via closed formulas.  
Contrary to the single-integral case, where a body of works has emerged over the last decades, see e.g.~\cite{Dac08, DMbook} 
 and the references therein, relaxation in the nonlocal setting is still largely unsolved. 
In the following, we give some background and outline briefly the latest developments related to this problem. 

The first paper to present a characterization of $L^p$-weak lower semicontinuity of $I_W$ in the scalar case $m=1$ goes back to Pedregal \cite{Ped97} in the late 1990s.  
Separate convexity of $W$ as a necessary and sufficient condition was identified almost ten years later in \cite{BeP06}, and generalized to the case of vector-valued fields, meaning for $m\geq 1$, in~\cite{Mun09}. More recent results, in particular on the inhomogeneous setting, can be found in \cite{BMC18, Ped16}.  

Motivated by these findings, a natural 
first guess for the relaxed functional associated with $I_W$ 
would seem to be a double integral with the separately convex hull $W^{\rm sc}$ of $W$ as integrand. 
However, there are one-dimensional counterexamples to disprove this conjecture, see e.g.~\cite[Example 3.1]{BeP06} or~\cite[Example 7.2]{BMC18} for integral functionals involving suitably chosen double integrands with eight or six wells, respectively.  Here, Corollary~\ref{cor:neq} and Corollary~\ref{prop:ness}, which both provide different necessary conditions for the relaxation of $I_W$ via separate convexification of $W$, put us in the position to generate a whole class of counterexamples.  Among the simplest ones for $m=1$ are the cases when $W$ is a four-well integrand with minima in 
\begin{align}\label{Wdist}
 \{(1, 0), (-1, 0), (0, 1), (0, -1)\}.
\end{align}

In~\cite{Ped16}, Pedregal claims even more than $I_W^{\rm rlx}\neq I_{W^{\rm sc}}$, namely that $I_W^{\rm rlx}$ may not be representable as a double integral at all. His reasoning is based on a monotonicity argument along the lines of a basic observation for single integrals. As  Bellido \& Mora-Coral point out in~\cite[Section~7]{BMC18}, though, this argument is in general not valid in the nonlocal context, see~Section~\ref{subsec:order} for more details. 

In this paper, we  
present two different proofs 
 to confirm that Pedregal's statement is indeed correct (see Propositions~\ref{prop:counterexample} and ~\ref{prop:counterexample2}). Both approaches 
involve the construction of a counterexample arising from a functional $I_W$, where $W$ is a double integrand of distance type, precisely, 
 \begin{align*}
W(\xi, \zeta) =  \dist^p ((\xi, \zeta), K) \quad \text{for $(\xi, \zeta)\in \R\times \R$,}
 \end{align*} 
 with a suitable combination of a compact set $K\subset \R\times \R$ and a norm inducing the distance; 
notice that it suffices to discuss the one-dimensional setting, since counterexamples in the case $m>1$ follow after a simple modification, cf. Remark~\ref{rem:generalization}.

In Proposition~\ref{prop:counterexample}, we take $K$ as in~\eqref{Wdist} and choose the $1$-norm on $\R\times \R$. 
Assuming to the contrary that $I_W^{\rm rlx}$ is a double integral with integrand $G$, we show that the infimum of $I_G$ is then attained in the origin, and thus $\min_{u\in L^p(\Omega)} I_G(u)= I_G(0,0)=0$; this follows from comparison arguments for nonlocal integral functions as established in Section~\ref{subsec:order}, which sets $G$ in relation to $W$ and $W^{\rm sc}$, and from exploiting that $G$ is separate convex as the double integrand of a weakly lower semicontinuous functional. On the other hand, it turns out that $\inf_{u\in L^p(\Omega)} I_W^{\rm rlx} > |\Omega|^2 \min_{(\xi, \zeta)\in \R\times \R} W(\xi, \zeta)=0$.
The proof is inspired by recent insights into the properties of nonlocal supremal functionals~\cite{KrZ19}.
Especially  the operation of diagonalization of sets  
in the sense of Definition~\ref{def:diagonalization}, applied here to the zero sublevel sets of $W$, and its interplay with approximate nonlocal inclusions, plays a central role; in fact, the latter are invariant under diagonalization as we prove in Theorem~\ref{prop:approx_inclusion}.

 The second counterexample in Proposition~\ref{prop:counterexample2} uses for $K$ the boundary of the convex hull of~\eqref{Wdist}, or equivalently, the boundary of the unit ball in the $1$-norm. In this case, the order relations from Section~\ref{subsec:order} allow us to conclude that if $I_W^{\rm rlx}$ is a double integral, then its integrand needs to coincide with $W^{\rm sc}$. However, we can prove that the correlation between the values $0$ and $1$, which is connected to the fact that $K$ fails to be diagonal, gives rise to $I_{W^{\rm sc}}(v) \neq I_W^{\rm rlx}(v)$
 for any non-constant $v:\Omega\to \{0,1\}$.

Closely related the double integrals we investigate here are nonlocal supremal functionals
 \begin{align}\label{supremal}
L^\infty(\Omega;\R^m)\ni u\mapsto {\rm esssup}_{(x,y)\in \Omega\times \Omega}\, Z(u(x), u(y))
\end{align}
with a suitable symmetric supremand $Z:\R^m\times \R^m\to \R$. 
Indeed, the latter arise (formally) through $L^p$-approximation, that is, in the limit process $p\to \infty$. 
The problem of relaxing~\eqref{supremal} has been settled recently in the case $m=1$ (and for general $m>1$ under an additional technical assumption, see~\cite[Remark 7.6]{KrZ19}): it is shown in~\cite[Theorem~1.3]{KrZ19} that the relaxation of~\eqref{supremal} with $m=1$ is structure preserving, meaning that it is again of supremal form, and that the relaxed supremand corresponds to the separate level convexification of the diagonalization of $Z$, cf.~\eqref{def:diagonalization2}. Here, in contrast, the challenging open question remains: What kind of representation for $I_W^{\rm rlx}$ in~\eqref{IWrlx} can be expected if double integrals are out of the picture? For first steps towards a better understanding, we refer to the Young measure relaxation result in \cite[Theorem 6.1]{BMC18}, as well as to 
 Proposition~\ref{prop:fourwell}, where we contribute a partial result by giving a closed formula for the relaxation of a specific class of double integrals. 

This article is organized as follows. After introducing notation and collecting some auxiliary results in Section~\ref{sec:preliminaries}, Section~\ref{sec:nonlocalinclusions} is concerned with the asymptotic behavior of approximate nonlocal inclusions; in particular, we provide a characterization of Young measures generated by sequences of nonlocal fields of the form $(u(x), u(y))$ for $(x,y)\in \Omega\times \Omega$ subject to approximate pointwise constraints, see Theorem~\ref{theo:Young}. 
Even though these results serve here primarily as technical tools for the remaining paper, they are also interesting in their own right. 
In Section \ref{sec:nesscon}, we address the issue of order relations and comparison arguments for double integrals as in~\eqref{I_W}, and deduce conditions on $W$ that are necessary for the identity $I_W^{\rm rlx} = I_{W^{\rm sc}}$. 
The centerpiece of this paper, namely the two counterexamples to structure preservation during relaxation, are presented, along with their proofs, in Section~\ref{sec:counterexample}. 
For a complementary viewpoint, we close in Section~\ref{sec:examples_relaxation} by discussing  functionals with double integrands in the form of distances to Cartesian sets and extended-valued indicators; the relaxations in both cases give rise to the intuitively expected double integrals with separately convexified integrands.  

\section{Notation and preliminaries}\label{sec:preliminaries}
To make the paper self-contained, we fix notation and collect some well-known results that will be used later on. 

\subsection{Notation} 
We denote the Euclidean norm of a vector $\eta=(\eta_1, \ldots, \eta_d)\in \mathbb R^d$ by $|\eta|= (\sum_{i=1}^d\eta_i^2)^{\frac{1}{2}}$, and use the notation $\|\cdot\|$ for a generic norm on $\R^m\times \R^m$ (without explicit mention, we often idenitfy $\R^m\times \R^m$ with $\R^{2m}$); specific choices of norms in the following include the $1$-norm $\|(\xi, \zeta)\|_1:=|\xi| + |\zeta|$, the Euclidean norm $\|(\xi, \zeta)\|_2=\sqrt{ |\xi|^2 + |\zeta|^2}$, or more generally the $q$-norm $\|(\xi, \zeta)\|_q=( |\xi|^q + |\zeta|^q)^{1/q}$ with $1\leq q<\infty$, and the maximum norm $\|(\xi, \zeta)\|_\infty = \max\{|\xi|, |\zeta|\}$ for $(\xi, \zeta)\in \R^m\times \R^m$.  
Further, $B_r(\xi, \zeta)\subset \R^m\times \R^m$ represents the closed ball of radius $r>0$ centered at $(\xi, \zeta)$, and for the distance of a point $(\xi, \zeta)\in \R^m\times \R^m$ to a non-empty, compact set $K\subset \R^m\times \R^m$, we write 
\begin{align}\label{dist}
\dist((\xi, \zeta), K) = \min_{(\alpha, \beta)\in K} \|(\xi, \zeta) -(\alpha,\beta)\|; 
\end{align}
if relevant, the use of a specific norm is indicated by super- and subscript indices, e.g.~$B_3^1(0,0) = \{(\xi, \zeta)\in \R^m\times \R^m: \|(\xi, \zeta)\|_1\leq 3\}$ or $\dist_\infty(\cdot, K) = \min_{(\alpha, \beta) \in K} \|\cdot - (\alpha, \beta)\|_\infty$. 
The generalized closed interval $[\xi, \zeta]$ with $\xi, \zeta\in \R^m$ is the set $\{\lambda\xi + (1-\lambda)\zeta\in \R^m: \lambda\in [0,1]\}$. 

For the complement of $A\subset \R^{d}$, we write $A^c=\R^{d}\setminus A$, whereas $A^{\rm co}$ stands for the convex hull of A. 
Let $\mathbbm{1}_A$ be the characteristic function of $A$, i.e., 
\begin{equation*}
\mathbbm{1}_A (\eta):=\left\{
\begin{array}{ll} 1 &\hbox{ if $\eta\in A$},\\
0 &\hbox{ otherwise,}
\end{array}
\right.\qquad \eta\in \R^d.
\end{equation*}

To refer to the minimum of a function $f:\R^d\to \R$ (if existent), we usually use the short-hand notation $\min f$ rather than $\min_{\eta\in \R^d} f(\eta)$. 

For any probability measure $\mu\in \mathcal Pr(\R^d)$, 
\begin{align*}
[\mu] := \langle \mu, {\rm id}\rangle = \int_{\R^d} \eta\, d\mu(\eta)
\end{align*} 
stands for its barycenter.
The product measure of $\nu, \mu\in \mathcal Pr(\R^d)$ is denoted by $\nu\otimes \mu$, and for the Lebesgue measure of a Lebesgue measurable set $U\subset\mathbb R^l$, we write ${\mathcal L}^l(U)$, or simply $|U|$. 
We employ standard notation for $L^p$-spaces with $p\in [1, \infty]$; particularly, our way to symbolize weak and weak$^\ast$ convergence of a sequence $(u_j)_j\subset L^p(U;\R^d)$
to a function $u\in L^p(U;\R^d)$ as $j\to \infty$ is 
 $u_j\weakly u$ in $L^p(U;\R^d)$ if $p\in [1, \infty)$ and $u_j \weaklystar u$ in $L^\infty(U;\R^d)$ if $p=\infty$.  Moreover, $S^\infty(U;\R^d)$ refers to the set of simple functions on $U$ with values in $\R^d$.

 Unless stated otherwise, $\Omega$ is a non-empty, open and bounded subset of $\mathbb R^n$ and $p>1$.

 \subsection{ A tool from convex analysis}
The following lemma is a corollary of a standard result in convex analysis, also known as zig-zag lemma (see e.g.~\cite[Lemma 20.2]{DMbook}). For the readers' convenience, we give here a simple explicit construction.

\begin{Lemma}\label{lem:oscillations}
Let $A\subset \R^m$ and suppose that $v\in S^\infty(\Omega;\R^m)$ is a simple function with image in $A^{\rm co}$. Then there exist a sequence $(v_j)_j\subset S^\infty(\Omega;\R^m)$ such that $v_j\in A$ a.e.~in $\Omega$ for all $j\in \mathbb N$ and $v_j\weaklystar v$ in $L^\infty(\Omega;\R^m)$. 
\end{Lemma}
\begin{proof}
Let $v=\sum_{i=1}^{N} \xi^{(i)}\mathbbm{1}_{\Omega^{(i)}}$
with $\xi^{(i)}\in A^{\rm co}$ and $\Omega^{(i)}$ disjoint measurable subsets of $\Omega$.

By Caratheodory's theorem (see e.g.~\cite[Theorem~2.13]{Dac08}), each $\xi^{(i)}\in A^{\rm co}$ is the convex combination of $m+1$ elements of $A$, that is, 
$\xi^{(i)} = \sum_{l=1}^{m+1}\lambda_l^{(i)} \xi_l^{(i)}$ with $\xi_l^{(i)}\in A$ and $\lambda_l^{(i)}\in [0,1]$ such that $\sum_{l=1}^{m+1} \lambda_l^{(i)}=1$.  

For any $i\in \{1, \ldots, N\}$ and $l\in \{1, \ldots, m+1\}$, let $\Omega_{l,j}^{(i)}$ with $j\in \mathbb N$ be measurable subsets of $\Omega^{(i)}$ such that  
\begin{align*}
\mathbbm{1}_{\Omega^{(i)}_{l,j}}\weaklystar \lambda_l ^{(i)}\mathbbm{1}_{\Omega^{(i)}}\quad\text{ in $L^\infty(\Omega)$ as $j\to \infty$;}
\end{align*}
this can be achieved for instance by choosing 
\begin{align*}
\Omega_{l,j}^{(i)} = \Omega^{(i)}\cap \bigcup_{z\in \mathbb Z^m} \frac{1}{j}z+\frac{1}{j}\Bigl[0,\sqrt[m]{\lambda_l^{(i)}}\Bigr]^m.
\end{align*} 
Then,
\begin{align*}
v_j^{(i)}:=\sum_{l=1}^{m+1} \xi_l^{(i)}\mathbbm{1}_{\Omega_{l,j}^{(i)}} \weaklystar v\quad\text{ in $L^\infty(\Omega;\R^m)$ as $j\to \infty$.}
\end{align*} 

 With these definitions, the sequence $(v_j)_j$ given by $v_j=\sum_{i=1}^N v_j^{(i)}\mathbbm{1}_{\Omega^{(i)}}$ for $j\in \mathbb N$
has all the desired properties.
\end{proof}

\subsection{Separate (level) convexity of sets and functions}
Convexity notions including separate convexity, separate level convexity and the related envelopes are a recurring theme in this paper. We briefly collect here some basics, referring the reader to \cite[Sections 2, 3 and 4]{KrZ19} for more properties, relations and characterizations of the following definitions. 

 A set $E\subset \R^m\times \R^m$ is called separately convex (with vectorial components), if for every $t\in (0,1)$ and every $(\xi_1, \zeta_1), (\xi_2, \zeta_2)\in E$ with $\xi_1=\xi_2$ or $\zeta_1=\zeta_2$ it holds that
	\begin{align*}
	t(\xi_1, \zeta_1) + (1-t)(\xi_2, \zeta_2) \in E.
	\end{align*}  
	The smallest separately convex set in $\mathbb R^m\times \R^m$ containing $E$ is called the separately convex hull of $E$ and denoted by $E^{\rm sc}$.

 Observe that for any $A\subset \R^m$, 
	\begin{align}
	\label{AcoxAco}
	(A \times A)^{\rm sc} =  (A \times A)^{\rm co} = A^{\rm co}\times A^{\rm co},
	\end{align}
	as a consequence of Carath\'eodory's theorem. 
	
The next definition introduces separate convexity for functions, as well as the weaker notion of separate level convexity. For the latter, recall that a function $f: \mathbb R^d \to \R_\infty:=\R\cup\{\infty\}$ is level convex
	if all (sub)level sets of $f$, that is,
	\begin{align*}
	L_c(f) : = \{\eta\in \R^d: f(\eta)\leq c\} \quad\text{ with $c\in \R$,}
\end{align*}
	are convex.

\begin{Definition} 
	We call $W:\mathbb R^m \times \mathbb R^m \to \R_\infty$ separately convex (with vectorial components) if for every $\xi \in \mathbb R^m$, the functions $W(\cdot, \xi)$ and $W(\xi,\cdot)$ are convex. The function $W$ is {separately level convex} (with vectorial components) if the sets
 $L_c(W) = \{(\xi,\eta)\in\mathbb R^m\times \mathbb R^m: W(\xi,\eta)\leq c\}$ are separately convex for all $c\in\R$. 
\end{Definition}
	
	Moreover, $W^{\rm sc}:\R^m\times \R^m\to \R_\infty$ ($W^{\rm slc}:\R^m\times\R^m\to \R_\infty$) stands for the separately (level) convex envelope of $W$, that is, the largest separately (level) convex function below $W$. Due to the implications between these different notations of convexity, it holds that 
\begin{align}\label{2.3}
W(\xi,\zeta)\geq W^{\rm slc}(\xi,\zeta)\geq W^{\rm sc}(\xi,\zeta)\geq W^{\rm co}(\xi,\zeta)
	\end{align}
for $(\xi, \zeta)\in \R^m\times \R^m$.

\begin{Lemma}
	\label{lemmaWslc}
	Let $V, W :\mathbb R^m \times \mathbb R^m\to \mathbb R$ be such that $V$ is separately level convex and 
	\begin{align}\label{Wslc}
	L_c(V)=  L_c(W)^{\rm sc}	
	\end{align} 
	for every $c \in \mathbb R$.
	Then $V=W^{\rm lsc}$.
	\end{Lemma}
\begin{proof}
 Due to $L_c(W^{\rm slc})\supset L_c(W)^{\rm sc}\supset L_c(W)$ for every $c\in \R$, it follows from~\eqref{Wslc} that  
		$$
		W^{\rm slc}\leq V\leq W.
		$$
		The separate level convexity of $V$ concludes the proof.
\end{proof}

Next, we discuss functions of distance type, meaning, $W:\R^m\times \R^m\to \R$ given by
\begin{align}\label{W_distance}
W(\xi, \zeta)=\dist^p((\xi, \zeta), K)\quad\text{ for $(\xi, \zeta)\in \R^m\times \R^m$,}
\end{align} 
where $K\subset \R^m \times \R^m$ is a non-empty and compact set and $p\geq 1$, cf.~also~\eqref{dist}. It is a classical implication of  Caratheodory's theorem that the convex envelope of $W$ as in~\eqref{W_distance} is 
\begin{align}\label{Wco_distance}
W^{\rm co}(\xi, \zeta) = \dist^p((\xi, \zeta), K^{\rm co})
\end{align}
for $(\xi, \zeta)\in \R^m\times \R^m$.

The following lemma presents a class of distance-type functions whose the separately convex envelope is convex.

\begin{Lemma}\label{lem:Wsc=Wco}
If $W$ is as in~\eqref{W_distance} 
with $K^{\rm sc}=K^{\rm co}$, then $W^{\rm sc}=W^{\rm co}$.
\end{Lemma}

\begin{proof}
We set $V=\dist^p(\cdot, K^{\rm sc})$. Since the separately convex and the convex hulls of $K$ coincide, one knows from~\eqref{Wco_distance} that $V=W^{\rm co}$ is convex. 
On the other hand, 
\begin{align*}
L_c(V) = K^{\rm sc} + B_{c^{1/p}}(0,0)= \bigl(K+B_{c^{1/p}}(0,0)\bigr)^{\rm sc} = L_c(W)^{\rm sc}
\end{align*}
for all $c\in \R$, and hence, $V=W^{\rm slc}$ by Lemma~\ref{lemmaWslc}. Summing up, this shows that $W^{\rm sc}=V=W^{\rm co}$ in view of~\eqref{2.3}.
\end{proof}

\subsection{The concept of diagonalization}\label{subsec:diagonalization} We recall some terminology related to the diagonalization of symmetric subsets of $\R^m\times\R^m$ as introduced in~{\cite[(4.1)]{KrZ19}}. A set $E\subset \R^m\times \R^m$ is symmetric if $(\xi,\zeta)\in E$ if and only if $(\zeta, \xi)\in E$. 

\begin{Definition}\label{def:diagonalization}
Let $E\subset \R^m\times \R^m$ be symmetric, then
\begin{align}\label{Khat}
\widehat{E} = \{(\xi, \zeta)\in E: (\xi, \xi), (\zeta, \zeta)\in E\} \subset \R^m\times \R^m
\end{align}
is called the diagonalization of $E$. We also use the alternative notation $E^\wedge$. 
\end{Definition}

Note that for any symmetric $E\subset \R^m\times \R^m$,
\begin{align}\label{EhatP}
\widehat{E}=\bigcup_{P\in \mathcal P_E}P;
\end{align} 
here, $\mathcal P_E$ stands for the set of maximal Cartesian subsets of $E$. A set $P\subset E$ is a maximal Cartesian subset of $E$ if $P=A\times A$ with $A\subset \R^m$ and if for any $B\subset \R^m$ with $A\subset B$ and $B\times B\subset E$ it holds that $B=A$. As a simple consequence of the definitions above, 
\begin{align}\label{PK=PKhat}
\mathcal P_E=\mathcal P_{\widehat E}
\end{align}
Moreover, if $K\subset \R^m\times\R^m$ is symmetric and compact, then also $\widehat K$ is compact. 

\subsection{Double integrals} We associate with any suitable $W:\R^m\times \R^m\to \R_\infty$ the double-integral functional $I_W:L^p(\Omega;\R^m)\to \R_\infty$ with
\begin{align}\label{I_W2}
I_W(u) :=\int_\Omega\int_\Omega W(u(x), u(y))\, dx\, dy
\end{align}
for $u\in L^p(\Omega;\R^m)$. To keep notations light, we dispense with highlighting explicitly the dependence on $p$ and $\Omega$, which will always be clear from the context. 

The nonlocal field $v_{w}\in L^p(\Omega\times \Omega;\R^m\times \R^m)$ corresponding to a function $w\in L^p(\Omega;\R^m)$ with $p\geq 1$ is defined as
\begin{align}\label{vw}
v_w(x,y) := (w(x), w(y))\quad \text{ for a.e.~$(x,y)\in \Omega\times \Omega$. }
\end{align} 

\section{Approximate nonlocal inclusions}\label{sec:nonlocalinclusions}

Throughout this section, let  $E\subset \R^m\times \R^m$ be symmetric.  

As in~\cite{KrZ19}, all essentially bounded solutions $u:\Omega\to \R^m$ to the (exact) nonlocal inclusion 
\begin{align}\label{inclusion3}
(u(x), u(y))\in E\quad \text{ for a.e.~$(x,y)\in \Omega\times \Omega$}
\end{align}
are collected in the set $\mathcal A_E$. In view of~\eqref{vw},  
\begin{align*}
\mathcal{A}_E=\{u\in L^\infty(\Omega;\R^m): v_u\in E \text{ a.e.~in $\Omega\times \Omega$}\},
\end{align*} 
and we introduce
 \begin{align}\label{AcalK}
\mathcal{A}_{E}^\infty :=\{u\in L^\infty(\Omega;\R^m): u_j\weaklystar u \text{ in $L^\infty(\Omega;\R^m)$ with $(u_j)_j \subset \mathcal A_E$}\}
\end{align} 
to describe the limiting behavior of sequences in $\mathcal A_E$. 
Upon relaxing the strict requirement of the exact nonlocal inclusion~\eqref{inclusion3}, one obtains an approximate version whose asymptotics is encoded in 
\begin{align}\label{BcalK_infty}
\begin{array}{l}
\mathcal{B}_{E}^\infty :=\{u\in L^\infty(\Omega;\R^m): u_j\weaklystar u \text{ in $L^\infty(\Omega;\R^m)$ with $(u_j)_j \subset L^\infty(\Omega;\R^m)$ such that} \\[0.2cm] \hspace{4.2cm} \text{${\rm dist}(v_{u_j}, E) \to 0$ in measure as $j\to \infty$}\}.
\end{array}
\end{align} 

Clearly, $\mathcal A_E^\infty\subset \mathcal B_{E}^\infty$. 
Under the additional assumption of compactness, we show equality of these two sets and provide a new characterization, valid in any dimension. 

\begin{Theorem}\label{prop:approx_inclusion} 
Let $K \subset \mathbb R^m \times \mathbb R^m$ be symmetric and compact. Then, 
\begin{align}\label{Acal_Kinfty}
\mathcal A_K^\infty= \mathcal B_K^\infty 
= \{u\in L^\infty(\Omega;\R^m): \text{$u\in A^{\rm co}$ a.e.~in $\Omega$ with $A\times A\in \mathcal P_K$}\},
\end{align}
recalling that $\mathcal P_K$ is the set of maximal Cartesian subsets of $K$.
\end{Theorem}

\begin{Remark}\label{rem:inclusions}
a) The sets $\mathcal A_K^\infty$ and $\mathcal B_K^\infty$ remain unchanged under diagonalization of $K$, that is, $\mathcal A_K^\infty=\mathcal A_{\widehat K}^\infty$ and $\mathcal B_{K}^\infty=\mathcal B_{\widehat K}^\infty$. 
Since  $\mathcal P_K=\mathcal P_{\widehat K}$ by~\eqref{PK=PKhat}, this is apparent from the representation~\eqref{Acal_Kinfty}.

Even though based on a different argument, the diagonalization invariance of $\mathcal A_K^\infty$ has been observed before in~\cite{KrZ19}; indeed,~\cite[Proposition~5.1]{KrZ19} yields that $\mathcal A_K=\mathcal A_{\widehat K}$, which implies $\mathcal A_K^\infty=\mathcal A_{\widehat K}^\infty$ in view of~\eqref{AcalK}. 

b) An equivalent way of expressing the characterization formula in~\eqref{Acal_Kinfty} is
\begin{align}\label{rem123}
\mathcal A_K^\infty = \mathcal B_K^\infty = \bigcup_{A\times A\in \mathcal P_K} \mathcal A_{A^{\rm co}\times A^{\rm co}} = \mathcal A_{\bigcup_{A\times A\in \mathcal P_K}A^{\rm co}\times A^{\rm co}},
\end{align}
cf.~also~\eqref{relax_indi}. Under an additional assumption on $K$, which is always satisfied for $m=1$, it was shown in~\cite[Theorem~1.1]{KrZ19} that 
\begin{align}\label{rem124}
\mathcal A_K^\infty=\mathcal A_{\widehat K^{\rm sc}},
\end{align} 
where $\widehat K^{\rm sc}$ is the separately convex hull of $\widehat K$. 

c) The assumption that the set $K$ in Theorem~\ref{prop:approx_inclusion} is closed cannot be dropped. To see this, we refer to~\cite[Remark~5.2]{KrZ19} for a simple example of a symmetric, non-closed set $E\subset \R^m\times \R^m$ and a set $\Omega\subset \R^m$ such that $\emptyset =\mathcal A_{\widehat E}\neq \mathcal A_{E}$. This implies in particular that $\mathcal A_E^\infty\neq \mathcal A_{\widehat E}^\infty$, and hence,~\eqref{Acal_Kinfty} cannot be true.
\end{Remark}

We postpone the proof of Theorem~\ref{prop:approx_inclusion} to the end of the section, since it is a consequence of the characterization of Young measures generated by sequences subject to approximate nonlocal constraints, which we   address next.

Following the notation of~\cite[Section~2.2]{KrZ19}, we consider 
the sets of parameterized measures
\begin{align}\label{YK}
\begin{array}{l}
\mathcal Y_E  := \{\Lambda \in L_w^\infty(\Omega\times \Omega; \mathcal Pr(\R^m\times \R^m)): \Lambda_{(x,y)}=\nu_x\otimes \nu_y \text{ with $\nu\in L_w^\infty(\Omega;\mathcal Pr(\R^m))$}  \\[0.2cm] 
\hspace{6.8cm} \text{and } {\rm supp\,} \Lambda_{(x,y)} \subset E \text{ for a.e.~$(x,y)\in \Omega\times\Omega$}\},
\end{array}
\end{align} 
\begin{align*}
\begin{array}{l} 
\mathcal Y_E^\infty :=\{\Lambda\in L_w^\infty(\Omega\times \Omega;\mathcal Pr(\R^m\times \R^m)): v_{u_j} \stackrel{YM}{\longrightarrow} \Lambda \text{ with $(u_j)_j\subset \mathcal A_E$}\},
\end{array}
\end{align*}
and
\begin{align}\label{YtildeE}
\begin{array}{l}
\widetilde{\mathcal Y}_{E}^\infty  :=\{\Lambda\in L_w^\infty(\Omega\times \Omega;\mathcal Pr(\R^m\times \R^m)):  v_{u_j} \stackrel{YM}{\longrightarrow} \Lambda \text{ with $(u_j)_j\subset L^\infty(\Omega;\R^m)$ such that} \\[0.2cm] \hspace{6.9cm}\text{${\rm dist}(v_{u_j}, E)\to 0$ in measure as $j\to \infty$}\};
\end{array}
\end{align}
here, $L_w^\infty(U;\mathcal Pr(\R^d))$ denotes the space of weakly measurable functions defined on an open set $U\subset \R^{l}$ with values in the space of probability measures on $\R^d$.  By $v_j \stackrel{YM}{\longrightarrow} \mu$, we mean that a sequence  $(v_j)_j\subset L^\infty(U;\R^d)$ generates the Young measure $\mu \in L^\infty_w(U;\mathcal Pr(\R^d))$, see e.g. \cite{FoL07, Ped97, Rin18} for more details. 

It was shown in~\cite[(5.21) and Theorem~5.11]{KrZ19} that for symmetric and compact $K\subset \R^m\times \R^m$,
\begin{align}\label{inclusionsYMold}
\bigcup_{P\in \mathcal P_{K}} \mathcal Y_P = \mathcal Y_K^\infty 
  \subset \widetilde{\mathcal Y}_{K}^\infty=
\mathcal Y_{K}.
\end{align}

In light of Proposition~\ref{lem:YM}, which is proven below, all four sets in~\eqref{inclusionsYMold} have to coincide. This gives rise to the following theorem. 

\begin{Theorem}\label{theo:Young}
Let $K\subset \R^m\times \R^m$ be compact and symmetric. Then 
\begin{align}\label{relation_YM}
\widetilde{\mathcal{Y}}_{K}^\infty 
=  \mathcal{Y}_K^\infty = \mathcal Y_{K} =\bigcup_{P\in \mathcal P_{K}}\mathcal Y_P.
\end{align}
\end{Theorem} 

Due to~\eqref{PK=PKhat}, all the sets in~\eqref{relation_YM} are invariant under diagonalization of $K$. In particular, $\mathcal Y_K=\mathcal Y_{\widehat K}$. 

 The next lemma serves as the main tool for the proof of Proposition~\ref{lem:YM}. 
\begin{Lemma}\label{lem:supp}
Let $\nu, \mu\in \mathcal Pr(\R^m)$ and $\Lambda=\nu\otimes \mu\in \mathcal Pr(\R^m\times \R^m)$. 
If $(\xi, \zeta), (\zeta, \xi), (\alpha, \beta), (\beta, \alpha) \in {\rm supp\,} \Lambda$, then 
\begin{align*}
\{\xi, \zeta, \alpha, \beta\}\times \{\xi, \zeta, \alpha, \beta\} 
\subset {\rm supp}\, \Lambda. 
\end{align*}
\end{Lemma}

\begin{proof}
It suffices to prove that one element of $\{\xi, \zeta, \alpha, \beta\}\times \{\xi, \zeta, \alpha, \beta\}$ different from $(\xi, \zeta)$, $(\zeta, \xi)$, $(\alpha, \beta)$ and $(\beta, \alpha)$ is contained in ${\rm supp\, }\Lambda$, say $(\xi, \alpha)$. For the other elements, the argument is analogous. 

Recalling the definition of the support of $\Lambda$, that is,
\begin{align*}
{\rm supp\,} \Lambda= \{(\xi, \zeta)\in \R^m\times \R^m: \Lambda(U)>0 \text{ for any open neighborhood $U$ of $(\xi, \zeta)$}\},
\end{align*}
let $U$ be an open neighborhood of $(\xi, \alpha)$.  
 Within $U$, one can find another neighborhood of the form $A\times B\subset U$ with $A, B\subset \R^m$ open such that $\xi\in A$ and $\alpha\in B$. From
\begin{align}\label{2}
\Lambda(U) \geq \Lambda(A\times B) = (\nu\otimes \mu)(A\times B) = \nu(A)\mu(B)>0,
\end{align}
we conclude that $(\xi, \alpha)\in {\rm supp\,}\Lambda$, as desired.
For the last inequality in~\eqref{2}, we have used that $A\times (B-\alpha+\zeta)$ is an open set containing $(\xi, \zeta)$, and thus, 
\begin{align*}
0<\Lambda(A\times (B-\alpha+\zeta)) = \nu(A)\mu(B-\alpha+\zeta)
\end{align*} 
by the assumption that $(\xi, \zeta)\in {\rm supp\,}\Lambda$.
This implies in particular that $\nu(A)>0$. Similarly, we show that $\mu(B)>0$. 
 \end{proof}
 
\begin{Proposition}\label{lem:YM}
Let $E\subset \R^m\times \R^m$ be symmetric. 
 Then, 
\begin{align*}
\mathcal Y_E =
\bigcup_{P\in \mathcal P_{E}} \mathcal Y_P.
\end{align*}
\end{Proposition}

\begin{proof} We prove the first two identities of 
\begin{align}\label{E=Ehat2}
\mathcal Y_E =\mathcal Y_{\widehat E} =\bigcup_{P\in \mathcal P_{\widehat E}} \mathcal Y_P =\bigcup_{P\in \mathcal P_{E}} \mathcal Y_P.
\end{align}
in separate steps; the last one is immediate, since $\mathcal P_E=\mathcal P_{\widehat E}$ by~\eqref{PK=PKhat}.

 \textit{Step~1: Invariance under diagonalization.} Since $\widehat E\subset E$, we only need to prove that $\mathcal Y_{E} \subset \mathcal Y_{\widehat E}$. Let $\Lambda \in \mathcal Y_{E}$, 
and assume to the contrary that 
 there exists a measurable set $N\subset \Omega\times \Omega$ with positive $\mathcal L^{2n}$-measure such that
\begin{align*}
{\rm supp\,} \Lambda_{(x,y)} \cap E\setminus \widehat E \neq \emptyset 
\end{align*}
for all $(x,y)\in N$. Due to the symmetry of $E$ and $\widehat E$, we may take $N$ to be symmetric. 

Now fix $(x, y)\in N$ and let $(\xi, \zeta)\in {\rm supp}\, \Lambda_{(x,y)}$ with $(\xi, \zeta)\notin \widehat E$.  Then also $(\zeta, \xi)\in {\rm supp}\, \Lambda_{(x,y)}$, and we infer from Lemma~\ref{lem:supp} that
\begin{align*}
\{\xi, \zeta\}\times \{\xi, \zeta\} \subset {\rm supp}\, \Lambda_{(x,y)} \subset E. 
\end{align*}
Hence, $(\xi, \zeta)\in \widehat E$ according to Definition~\ref{def:diagonalization}, which is  a contradiction.

\textit{Step~2: Alternative representation of $\mathcal Y_{\widehat E}$.}  By definition, any $P\in \mathcal P_{\widehat E}$ is contained in $\widehat E$; hence, $\bigcup_{P\in\mathcal P_{\widehat E}}\mathcal Y_P\subset \mathcal Y_{\widehat E}$ is immediate. 
For the reverse inclusion, let $\Lambda\in \mathcal Y_{\widehat E}$. To show that $\Lambda\in \mathcal Y_P$ for some $P\in \mathcal P_{\widehat E}$,
we argue again by contradiction,  assuming that there is a measurable set $N\subset \Omega\times \Omega$ with $\mathcal L^{2n}(N)>0$, as well as a maximal Cartesian set $\overline{P}\in \mathcal P_{\widehat E}$ such that for all $(x,y)\in N$, 
\begin{align*}
{\rm supp\,} \Lambda_{(x,y)} \subset \widehat E,
\end{align*}
as well as 
\begin{align*}
{\rm supp\,} \Lambda_{(x,y)} \cap \overline Q\neq \emptyset \quad \text{and}\quad {\rm supp\,} \Lambda_{(x,y)} \cap \underline{Q} 
\neq \emptyset,
\end{align*}
with
\begin{align}\label{Qbar}
\overline Q:= \overline{P}\setminus \bigcup_{P\in \mathcal P_{\widehat E}, P\neq \overline{P}}P 
\quad \text{and}\quad
\underline Q:= \bigcup_{P\in \mathcal P_{\widehat E}, P\neq \overline{P}}P \setminus \overline{P}.
\end{align} 
Since $\overline Q$ and $\underline Q$ are both symmetric, $N$ can be chosen to be symmetric, too.

Next, we fix $(x,y)\in N$ and take $(\xi, \zeta), (\alpha, \beta)\in {\rm supp\,} \Lambda_{(x,y)}$ such that 
\begin{align}\label{inclusionPPbar}
(\xi, \zeta)\in \overline Q \quad \text{and } (\alpha, \beta)\in \underline Q.
\end{align}
By symmetry, also $(\zeta, \xi), (\beta, \alpha)\in {\rm supp\,} \Lambda_{(x,y)}$, 
and we infer from Lemma~\ref{lem:supp} that
\begin{align*}
M:=\{\xi, \zeta, \alpha, \beta\}\times \{\xi, \zeta, \alpha, \beta\} \subset {\rm supp}\, \Lambda_{(x,y)} \subset \widehat E. 
\end{align*} 
Since $M$ is a Cartesian product, it is contained in some maximal Cartesian subset $P$ of $\widehat E$. However, in view of~\eqref{inclusionPPbar} and~\eqref{Qbar}, $P$ cannot coincide with any element of $\mathcal P_{\widehat E}$; indeed, $(\xi, \zeta)\in P$, but $(\xi, \zeta)$ does not lie in any maximal Cartesian subset of $\widehat E$ other than $\overline P$, hence, $P=\overline P$; on the other hand, $(\alpha, \beta)\in P$, but $(\alpha, \beta)\notin \overline P$, which shows that $P\neq\overline P$. This is the sought contradiction. 
\end{proof}

\begin{proof}[Proof of Theorem~\ref{prop:approx_inclusion}]
The equality of $\mathcal A_K^\infty$ and $\mathcal B_K^\infty$ follows from Theorem~\ref{theo:Young} when thinking in terms of barycenters of Young measures. Indeed, it suffices to use~\eqref{relation_YM} along with the observation that 
 \begin{align*}
 \mathcal A_{K}^\infty =  \{u\in L^\infty(\Omega;\R^m): v_u=[\Lambda], \ \Lambda\in \mathcal Y_{K}^\infty\} 
 \ \text{and}\  \mathcal \mathcal B_{K}^\infty = \{u\in L^\infty(\Omega;\R^m): v_u=[\Lambda], \ \Lambda\in \widetilde{\mathcal Y}_{K}^\infty\}.
\end{align*}
For the desired representation formula, we invoke again~\eqref{relation_YM} to deduce that
\begin{align*}
\mathcal A_K^\infty & = \mathcal B_K^\infty  = \Big\{u\in L^\infty(\Omega;\R^m): v_u=[\Lambda], \Lambda\in \bigcup_{P\in \mathcal P_K} \mathcal Y_P\Big\} \\ & = 
\bigcup_{A\times A\in \mathcal P_K} \{u\in L^\infty(\Omega;\R^m): v_u=[\Lambda], \Lambda\in \mathcal Y_{A\times A}\} \\ &  =   \bigcup_{A\times A\in \mathcal P_K} \{u\in L^\infty(\Omega;\R^m): u=[\nu], \nu\in L^\infty_w(\Omega; \mathcal Pr(\R^m)),\, { \rm supp\,} \nu_x\subset A \text{ for a.e.~$x\in \Omega$}\}  \\ &= \bigcup_{A\times A\in \mathcal P_K} \{u\in L^\infty(\Omega;\R^m): u\in A^{\rm co} \text{ a.e.~in $\Omega$}\},
\end{align*}
which was the claim. In the last step, we used the well-known characterization of convex hulls via barycenters of probability measures, see e.g.~\cite{Dol03}. 
\end{proof}

\begin{Remark}\label{rem:BKinftyp}
Let $p\geq 1$ and $E\subset \R^m\times \R^m$ symmetric. Replacing the weakly$^\ast$ converging $L^\infty$-sequences in the above definitions of $\widetilde{\mathcal Y}_E^\infty$ and $\mathcal B_E^\infty$ (see~\eqref{YtildeE} and~\eqref{BcalK_infty}) by weakly converging $L^p$-sequences results in new sets of functions and parametrized measures, which we want to call $\widetilde{\mathcal Y}_{E, p}^\infty$ and $\mathcal B_{E, p}^\infty$, respectively. 

If $K\subset \R^m\times \R^m$ is symmetric and compact, then 
\begin{align*}
\mathcal B_{K,p}^\infty = \mathcal B_K^\infty \qquad \text{ and }\qquad \widetilde{\mathcal Y}_{K,p}^\infty =\widetilde{\mathcal Y}_K^\infty,
\end{align*}
which is a consequence of~~\cite[Proposition~2.2]{Ped97} and the fundamental theorem on Young measures, see  e.g.~\cite[Theorem~8.6\,(iii)]{FoL07}. 
\end{Remark}

\section{Necessary conditions for relaxation via separate convexification}\label{sec:nesscon}
As pointed out in the introduction, each of the papers~\cite{BMC18, BeP06} presents a specific example of a double-integral functional of multi-well form for which separate convexification of the integrand fails in providing a correct relaxation formula.
Here, we generalize these findings and generate a whole class of such examples (see Corollary~\ref{cor:neq}), motivated by recent insights from the study of nonlocal supremal functionals and nonlocal inclusions~\cite{KrZ19}. A key ingredient is the following notion of diagonalization for functions introduced in \cite[(7.1)]{KrZ19}.

\begin{Definition}\label{def:diagonalization2}
The diagonalization of a symmetric function $W:\R^m\times \R^m\to \R$ is defined as 
\begin{align*}
\widehat{W}:\R^m\times \R^m\to \R, \quad \widehat{W}(\xi, \zeta) = \inf\{c\in \R: (\xi, \zeta)\in \widehat{L_c(W)}\},
\end{align*}
where $\widehat{L_c(W)}=L_c(W)^\wedge$ is the diagonalization of the sublevel set $L_c(W)$ with $c\in \R$ in the sense of Definition~\ref{def:diagonalization}.
\end{Definition}

Notice that $\widehat W\geq W$ and that 
\begin{align}\label{identity_LcW}
L_c(\widehat{W}) = \widehat{L_c(W)} = L_c(W)^\wedge
\end{align}
for all $c\in \R$, cf.~\cite[(7.2)]{KrZ19}. 


\subsection{Double integrals and order relations}\label{subsec:order} 
An important difference between the theory of single- and double-integral functionals with substantial conceptual and technical ramifications lies in the order relations for the functionals and their integrands. 

Whereas it holds for any suitable $f:\R^m\to \R$ that
\begin{align*}
\inf_{u\in L^p(\Omega;\R^m)} \int_{\Omega} f(u)\, dx \geq 0\quad \Rightarrow\quad  f\geq 0,
\end{align*}
the analogy of this implication is in general not true in the context of double integrals. In fact, if 
\begin{align*}
\inf_{u\in L^p(\Omega;\R^m)} \int_\Omega\int_\Omega W(u(x), u(y))\, dx\,dy \geq 0,
\end{align*} 
for a suitable $W:\R^m\times \R^m\to\R$, the integrand $W$ may take both positive and negative values, which is owed to nonlocal effects. 
This observation was pointed out first by Bellido \& Mora-Corral in~\cite[Section 7]{BMC18} and illustrated with an explicit scalar example of the form $W(\xi, \zeta)= w(\xi-\zeta)$ for $(\xi, \zeta)\in \R\times \R$, where $w:\R\to \R$ is a fourth-order even polynomial; see~\cite[Example~7.2]{BMC18} for the details. In the next proposition, we investigate a 
more general class of related integrands, cf.~also Remark~\ref{rem:1}\,c) below.
\begin{Proposition}\label{prop:geq}  Let $W:\R^m\times \R^m\to \R$ be a symmetric, lower semicontinuous function with $p$-growth and $p$-coercivity. If 
\begin{align}\label{Wmin}
\min W 
<
 \min \widehat W,
\end{align} then, 
\begin{align*}
|\Omega|^2 \min \widehat W  \geq \inf_{u\in L^p(\Omega;\R^m)} I_W(u) > |\Omega|^2\min W.
\end{align*}
\end{Proposition}

\begin{Remark}\label{rem:1}
a) It is clear that $W$ attains its infimum on $\R^m\times \R^m$ 
due to its coercivity and lower semicontinuity of $W$. Since these two properties carry over to $\widehat{W}$ considering that the sublevel sets of $\widehat W$ are again compact (cf.~\eqref{identity_LcW} and at the end of Section~\ref{subsec:diagonalization}),  also $\min \widehat{W}$ is well defined. 

b) Due to Proposition~\ref{prop:geq} and the properties of $\widehat W$, one finds that $\min W=\min \widehat W$ if and only if 
\begin{align*}
\inf_{u\in L^p(\Omega;\R^m)} I_W(u) = \min_{u\in L^p(\Omega;\R^m)} I_W(u) = |\Omega|^2 \min W.
\end{align*}

c) The statement of Proposition~\ref{prop:geq} is valid also for double integrands $W:\R^m\times \R^m\to \R$ of the form $W(\xi, \zeta)=w(\xi-\zeta)$ for $(\xi, \zeta)\in \R^m\times \R^m$, where $w:\R^m\to \R$ is a lower semicontinuous function with $p$-growth and $p$-coercivity, provided we consider $I_W$ only on the smaller space of $L^p(\Omega;\R^m)$-functions with vanishing mean value. 
\end{Remark}

\begin{proof}[Proof of Proposition \ref{prop:geq}]
For simplicity of notation, we write 
 $\inf I_W:=\inf_{u\in L^p(\Omega;\R^m)}I_W(u)$ in what follows, 
and we assume without loss of generality that $\min W=0$; otherwise, $W$ can be translated suitably. 

First, we show the estimate 
\begin{align}\label{est111}
|\Omega|^2  \min \widehat W\geq \inf I_W.
\end{align} 
Let $(\xi, \zeta)\in \R^m\times \R^m$ be a minimizer of $\widehat W$. Then, $(\xi, \zeta)\in L_{\min \widehat W} (\widehat W) = L_{\min \widehat W} (W)^\wedge$ by~\eqref{identity_LcW}, so that
\begin{align}\label{identity}
\widehat W(\xi, \zeta) = \widehat W (\zeta, \xi) = \widehat W(\xi, \xi) = \widehat W(\zeta, \zeta) = \min \widehat W,
\end{align}
by the symmetry of $\widehat{W}$ and the definition of diagonalization of sets in~\eqref{Khat}. 
Considering the constant function $v:\Omega\to \R^m$ given by $v(x)=\xi$ for $x\in \Omega$, we conclude in view of~\eqref{identity} and $W\leq \widehat W$ that
$$
\inf I_W \leq I_W(v) =  |\Omega|^2 W(\xi, \xi) \leq |\Omega |^2\widehat W(\xi,\xi)=|\Omega|^2{\rm min}\widehat W.
$$
This implies~\eqref{est111}. 

To prove the strict inequality $\inf I_W>  \min W=0$, we assume to the contrary that $\inf I_W=0$, meaning that there exists a sequence $(u_j)_j\subset L^p(\Omega;\R^m)$ such that 
\begin{align}\label{est908}
\lim_{j\to \infty} I_W(u_j)= \lim_{j\to \infty} \int_{\Omega}\int_\Omega W(v_{u_j}(x,y))\, dx\, dy = 0,
\end{align}
cf.~\eqref{vw}. 
As a consequence of the $p$-coercivity of $W$, $(v_{u_j})_j$ is uniformly bounded in $L^p(\Omega\times \Omega;\R^m\times \R^m)$. 

If $\Lambda=\{\Lambda_{(x,y)}\}_{(x,y)} =\{\nu_x\otimes \nu_y\}_{(x,y)}$ with $\nu\in L_w^\infty(\Omega; \mathcal Pr(\R^m))$ is the Young measure generated by a (non-relabeled) subsequence of $(v_{u_j})_j\subset L^p(\Omega\times \Omega;\R^m\times \R^m)$ 
according to \cite[Proposition~2.3]{Ped97}, the fundamental theorem on Young measures (see e.g.~\cite[Theorem~8.6\,(i)]{FoL07}) yields that 
\begin{align*}
\lim_{j\to \infty} \int_\Omega\int_\Omega W (v_{u_j}(x,y)) \, dx\, dy \geq \int_\Omega\int_\Omega \langle \Lambda_{(x,y)}, W\rangle \, dx\, dy, 
\end{align*}
where $\langle\Lambda_{(x,y)}, W\rangle := \int_{\R^m}\int_{\R^m}  W(\xi, \zeta)\, d\Lambda_{(x,y)}(\xi, \zeta)$.

In light of~\eqref{est908} and the non-negativity of $W$, 
it follows that $\langle \Lambda_{(x,y)}, W\rangle = 0$ for a.e.~$(x,y)\in \Omega\times \Omega$, and hence, 
${\rm supp\,}\Lambda_{(x,y)}\subset L_0(W)$ for a.e.~$(x,y)\in \Omega\times \Omega$,
or equivalently by~\eqref{YK}, 
\begin{align}\label{Lambdain}
\Lambda\in \mathcal Y_{L_0(W)}.
\end{align}
 
On the other hand, \eqref{E=Ehat2} in the proof of Lemma~\ref{lem:YM} together with \eqref{Wmin} results in
\begin{align}\label{emptyset}
 \mathcal Y_{L_0(W)} = \mathcal Y_{\widehat{L_0(W)}} = \emptyset. 
\end{align}

Combining~\eqref{Lambdain} with~\eqref{emptyset} produces the desired contradiction. 
\end{proof}

We continue our discussion of order relations for double integrals with the following basic, yet useful, observation.

\begin{Lemma}\label{lem:FG}
Let $V, W:\R^m\times \R^m\to \R$ be symmetric, lower semicontinuous integrands with $p$-growth such that $I_V\leq I_W$. 
Then $V(\xi, \xi)\leq W(\xi, \xi)$ for all $\xi\in \R^m$. 

Moreover, if $V(\xi, \xi)=W(\xi, \xi)$ for all $\xi\in A\subset \R^m$, 
then $V\leq W$ on  $A\times A$.
\end{Lemma}

\begin{proof}
Trivially, the first statement follows by evaluating $I_V$ and $I_W$ for constant functions.
To show the second statement, let
$(\xi, \zeta)\in A \times A$  and consider a piecewise constant function $v=\xi \mathbbm{1}_{\Omega_\xi} + \zeta \mathbbm{1}_{\Omega\setminus \Omega_\xi}\in S^\infty(\Omega;\R^m)$ with a measurable set $\Omega_\xi \subset \Omega$ such that $|\Omega_{\xi}|  =\frac{1}{2}|\Omega|$. Then, 
\begin{align*}
 &\frac{|\Omega|^2}{4} V(\xi, \xi) + \frac{|\Omega|^2}{4}V(\zeta,\zeta) + \frac{|\Omega|^2}{2} V(\xi, \zeta)= I_V(v)  \\ & \qquad\qquad   \leq I_W(v) = \frac{|\Omega|^2}{4} W(\xi, \xi) + \frac{|\Omega|^2}{4}W(\zeta,\zeta) + \frac{|\Omega|^2}{2} W(\xi, \zeta).
\end{align*}
Since $V$ and $W$ coincide on the diagonal elements in $A\times A$, this implies $V(\xi,\zeta)\leq W(\xi, \zeta)$, concluding the proof.
\end{proof}

\begin{Remark}
	\label{Elbauappl}
	The previous lemma shows in particular that a double integral $I_W$ as in~\eqref{I_W2} determines its integrand $W$ uniquely. We point out  that this is in contrast to the supremal setting, where according to~\cite[(7.3)]{KrZ19}, all supremands with the same diagonalization in the sense of Definition~\ref{def:diagonalization2} generate the same supremal functional. 
\end{Remark}

With the help of the previous lemma, we can derive the following bounds for a class of relaxed double integrands.

\begin{Proposition}
	\label{intrep}
	Let $W, G:\mathbb R^m\times \mathbb R^m\to \mathbb R$ be symmetric, lower semicontinuous functions with $p$-growth. Suppose that $I_W^{\rm rlx}=I_G$ and that there exists $A\subset \R^m$ such that $W(\xi,\xi)=W^{\rm sc}(\xi,\xi)$ for every $\xi \in A^c$. 
	Then, 
	\begin{align}\label{est80}
	W^{\rm sc}\leq G\leq W \quad \text{on $(A\times A)^c$.}
	\end{align} 
	If $A=\emptyset$, it holds that $G=W^{\rm sc}$.
	\end{Proposition}	
	
\begin{proof}[Proof] 
From
	\begin{align*}
	 I_{W^{\rm sc}}\leq I^{\rm rlx}_W=I_G\leq I_W,
	\end{align*}
we conclude with Lemma~\ref{lem:FG} that
$W^{\rm sc}(\xi, \xi)\leq G(\xi, \xi)\leq W(\xi, \xi)$ for all $\xi\in A^c$, 
which, in view of our hypothesis, yields
\begin{align}\label{G=Wdiagonal1}
W^{\rm sc}(\xi, \xi)=G(\xi, \xi) = W(\xi, \xi)\quad \text{ for $\xi\in A^c$. }
\end{align}

Let $(\xi, \zeta)\in (A\times A)^c$, and assume without loss of generality that $\xi \not \in A$; otherwise interchange the roles of $\xi$ and $\zeta$. Moreover, suppose for simplicity that $|\Omega|=1$. 
We define 
\begin{align*}
v=\xi \mathbbm{1}_{\Omega_\xi} + \zeta \mathbbm{1}_{\Omega\setminus \Omega_\xi}\in S^\infty(\Omega;\R^m), 
\end{align*}
where $\Omega_\xi\subset \Omega$ is measurable such that $|\Omega_{\xi}| =\lambda$ with $\lambda\in (0, 1)$. Then,
\begin{align*}
& \lambda^2 G(\xi, \xi) + (1-\lambda)^2 G(\zeta,\zeta) + 2\lambda(1-\lambda) G(\xi, \zeta)= I_G(v)\\ & \qquad \leq I_W(v) =\lambda^2 W(\xi, \xi) + (1-\lambda)^2W(\zeta,\zeta) + 2\lambda(1-\lambda) W(\xi, \zeta).
\end{align*}
Due to \eqref{G=Wdiagonal1}, this can be rewritten as
\begin{align*}
G(\xi, \zeta)\leq  \frac{1-\lambda}{2\lambda}\bigl(W(\zeta,\zeta) - G(\zeta, \zeta)\bigr) + W(\xi, \zeta).
\end{align*}
Letting $\lambda$ tend to $1$, allows us to conclude that $G\leq W$ on the complement of $A\times A$. 

If we replace $G$ in the argument above with $W^{\rm sc}$, and $W$ with $G$, the exact same reasoning provides that $W^{\rm sc}\leq G$ outside of $A\times A$. Overall, this proves~\eqref{est80}.

For the statement on the special case $A=\emptyset$, we infer from~\eqref{est80} that $W^{\rm sc}\leq G\leq W$. Since $G$ has to be separately convex due to the $L^p$-weak lower semicontinuity of $I_G$ (see e.g.~\cite[Theorem~1.1]{BeP06}), it follows that even $G\leq W^{\rm sc}$, which entails the claim. 
	\end{proof}

\subsection{Implications for relaxation formulas}
Based on the results of Section~\ref{subsec:order}, 
we derive necessary conditions for the relaxation of $I_W$ as in~\eqref{I_W2} via separate convexification of the double integrand $W$. We distinguish between the two cases when $\min W\neq\min \widehat W$ and $\min W= \min \widehat W$, addressed in Corollary~\ref{cor:neq}
 and Corollary~\ref{prop:ness}, respectively. 
 
\begin{Corollary}\label{cor:neq}
Let $W:\R^m\times \R^m\to \R$ be as in Proposition~\ref{prop:geq} with $\min \widehat W>\min W$.
If $I_W^{\rm rlx} = I_{W^{\rm sc}}$, then 
\begin{align}\label{con12}
\min \widehat{W^{\rm sc}} > \min W.
\end{align}
\end{Corollary}

\begin{proof}
	 By the assumptions and standard results in relaxation theory (see e.g.~\cite[Section 3]{DMbook}), one obtains that $$\inf_{u\in L^p(\Omega;\R^m)}I_W(u)
	=\min_{u\in L^p(\Omega;\R^m)}I^{\rm rlx}_W(u)= \min_{L^p(\Omega;\R^m)}I_{W^{\rm sc}}(u).$$
	Thus, applying Proposition~\ref{prop:geq} to $I_W$ and Remark~\ref{rem:1}\,b) with $\widehat{W^{\rm sc}}$ results in
$$
|\Omega|^2\min W< \inf_{u\in L^p(\Omega;\R^m)}I_W(u)=\min_{u\in L^p(\Omega;\R^m)}I_{W^{\rm sc}}(u) \leq\inf_{u\in L^p(\Omega;\R^m)}I_{\widehat{W^{\rm sc}}}(u) = |\Omega|^2\min \widehat{W^{\rm sc}},$$  
which, due to $|\Omega|>0$, concludes the proof. 
\end{proof}

For a simple one-dimensional example of a double integrand that fails to satisfy the necessary conditions~\eqref{con12}, see Section~\ref{section:ex1}, especially~\eqref{ref1a}.
Next, we formulate a corresponding result in the case when the minima of the double integrand and its diagonalization coincide.   

\begin{Corollary}\label{prop:ness}
	Let $W:\mathbb R^m\times \mathbb R^m\to \mathbb R$ be symmetric and lower semicontinuous with $p$-growth and $p$-coercivity such that $\min W=\min \widehat W =0$. If $I_W^{\rm rlx} = I_{W^{\rm sc}}$,  
then, 
\begin{align}\label{ness_dV}
L_0(W^{\rm sc})^\wedge = \bigcup_{A\times A\in \mathcal P_{L_0(W)}} A^{\rm co}\times A^{\rm co}. 
\end{align} 

If $m=1$, it holds that
\begin{align}\label{ness}
L_0(W^{\rm sc})^\wedge = 
( L_0(W)^\wedge)^{\rm sc}. 
\end{align} 
\end{Corollary}

\begin{proof}
According to the classical abstract theory of relaxation (see e.g.~\cite[Section~9]{Dac08} and~\cite[Section~3]{DMbook} and the references therein), the set of $L^p$-weak limits of sequences of almost minimizers for $I_W$ coincides with the set of minimizers of the relaxed functional $I_W^{\rm rlx} = I_{W^{\rm sc}}$. 

Translated into the language of nonlocal inclusions in the spirit of Section~\ref{sec:nonlocalinclusions}, this means 
\begin{align*}
 \mathcal B_{L_0(W), p}^\infty = 
 \mathcal A_{L_0(W^{\rm sc})}. 
\end{align*}
Because $L_0(W)$ is symmetric and compact as the sublevel set of a lower semicontinuous and coercive symmetric function, we infer from~\eqref{rem123} (cf.~also Theorem~\ref{prop:approx_inclusion}) in conjunction with Remark~\ref{rem:BKinftyp} that 
\begin{align*} 
\mathcal A_{L_0(W^{\rm sc})} =  \mathcal A_{L_0(W)}^\infty = 
\mathcal A_{\bigcup_{A\times A\in \mathcal P_{L_0(W)}}A^{\rm co}\times A^{\rm co}}.
\end{align*} 
By \cite[Proposition~5.1]{KrZ19}, this is equivalent to 
\begin{align}\label{276}
 L_0(W^{\rm sc})^\wedge = \Bigl(\bigcup_{A\times A\in \mathcal P_{L_0(W)}} A^{\rm co}\times A^{\rm co}\Bigr)^\wedge.
\end{align}
Since the right-hand side of~\eqref{276} is the union of Cartesian products and thus, already diagonal, this shows~\eqref{ness_dV}. 
	
Regarding $m=1$,  
we infer from~\eqref{rem124} that $\mathcal A_{L_0(W)}^\infty = \mathcal A_{(L_0(W)^\wedge)^{\rm sc}}$. Then,~\eqref{ness} follows from the same argumentation as above, along with the observation that diagonalization keeps $(L_0(W)^\wedge)^{\rm sc}$ invariant, see~\cite[Remark~4.8]{KrZ19}. 
\end{proof}

Given that the operations of diagonalization and separate convexification do not commute,~\eqref{ness_dV} and \eqref{ness} impose in general non-trivial restrictions on $W$, as the following example for $m=1$ illustrates.

\begin{Example} 
Let $K=\{(\pm 1, 0), (0, \pm 1), (2, 2)\}\subset \R\times \R$ and consider
\begin{align*}
W(\xi, \zeta)={\rm dist}^p ((\xi, \zeta), K)\quad \text{ for $(\xi, \zeta)\in \R\times \R$,}
\end{align*}
with respect to any norm on $\R\times \R$. 
On the one hand, $L_0(W)^\wedge = \widehat K= \{(2,2)\}$ is already separately convex. On the other hand, 
\begin{align*}
L_0(W^{\rm sc})\supset L_0(W)^{\rm sc} = K^{\rm sc} =\{0\}\times [-1,1] \cup [-1,1]\times \{0\}\cup \{(2,2)\},
\end{align*} 
which, after diagonalization, turns into $L_0(W^{\rm sc})^\wedge\supset \widehat{K^{\rm sc}} = \{(0,0), (2,2)\}$. This shows that~\eqref{ness} is not satisfied here. 
\end{Example}
\section{Counterexamples for preservation of double-integral character}\label{sec:counterexample}

In this section, we present and analyze two examples to disprove that the relaxation of a double integral yields again a double integral. In both cases, the integrand $W$ of $I_W$ as in~\eqref{I_W} with $m=1$ measures the distance to a compact set $K\subset \R\times \R$ of diamond-like shape; choosing the $1$-norm is a good fit with the structure of $K$ and helps to keep the calculations simple. The conceptual difference between the two densities is that for the first, $\min W\neq\min \widehat W$, while the second satisfies $\min W=\min \widehat W$. 

\subsection{First counterexample}\label{section:ex1}
Let $W:\R\times \R\to \R$ be the function defined by
\begin{align}\label{Wexample}
W(\xi, \zeta) = {\rm dist}^p_1((\xi, \zeta), K) 
\end{align} 
for $(\xi, \zeta)\in \R\times \R$ with $K=\{(\pm 1, 0), (0, \pm 1)\}$ and the underlying norm $\|\cdot \|_{1}$. 
We start by observing that 
$K^{\rm sc} =  \{0\}\times [-1,1]\cup [-1,1]\times \{0\}$. 

Looking into the sublevel sets of $W$ and its different convex envelopes gives the following relations: For $c\in \R$, 
\begin{align}\label{sublevels}
L_c(W)= \begin{cases}
B^1_{1+c^{1/p}}(0,0) & \text{for $c\geq 1$,}\\
\bigcup_{(\xi,\zeta)\in K} B^1_{c^{1/p}}(\xi, \zeta) & \text{for $0\leq c\leq 1$,}\\
\emptyset & \text{for $c<0$,}
\end{cases}
\end{align}
and 
\begin{align}\label{sublevels2}
 & L_c(W^{\rm sc}) \supset L_c(W^{\rm slc})\supset  L_c(W)^{\rm sc} \nonumber\\ &\qquad\qquad  = \begin{cases}
 L_c(W) & \text{for $c\geq 1$,}\\
L_c(W)\cup [(-c^{1/p}, c^{1/p})\times (-1,1) ]\cup [(-1,1)\times (-c^{1/p}, c^{1/p}) ] & \text{for $0\leq c\leq 1$,}\\
\emptyset & \text{for $c<0$,}
\end{cases} 
\end{align}
see also Figure~\ref{fig:levelsets}; by \eqref{Wco_distance}, 
$W^{\rm co}={\rm dist}^p_1(\cdot, K^{\rm co})$, and hence, 
\begin{align}\label{lcWco}
L_c(W^{\rm co}) = \begin{cases}
B^1_{1+c^{1/p}}(0,0)  &  \text{ for $c\geq 0$},\\
\emptyset & \text{ for $c<0$.}
\end{cases}
\end{align}
\begin{figure}\label{fig:levelsets}
\begin{tikzpicture}
\begin{scope}[xshift = -4cm, yshift = 0cm, scale = 0.6]

\draw[->, thick] (0,-6.5) -- (0,6.5);
\draw[->, thick] (-6.5,0) -- (6.5,0);
\draw (-0.5,6.1) node {$\zeta$};
\draw (6.0,-0.5) node{$\xi$};
\draw (2.0, -0.5) node{\small $1$};
\draw (2.0,-0.2) -- (2.0, 0);


\draw[violet, fill=violet] (-2,0) circle (0.1cm);
\draw[violet, fill=violet] (2,0) circle (0.1cm);
\draw[violet, fill=violet] (0,-2) circle (0.1cm);
\draw[violet, fill=violet] (0, 2) circle (0.1cm);
\draw[thick, orange]  (-5, 0) --  (0,-5) -- (5,0)  -- (0,5) -- (-5,0); 
\draw[thick, yellow]  (-6, 0) --  (0,-6) -- (6,0)  -- (0,6) -- (-6,0);
\draw[thick, red]  (-4, 0) --  (0,-4) -- (4,0)  -- (0,4) -- (-4,0); 
\draw[thick, red]  (-2, -2) --  (2,2); 
\draw[thick, red]  (-2, 2) --  (2,-2); 

\draw[thick, purple]  (-3, 0) --  (-2,-1) -- (-1,0)  -- (-2,1) -- (-3,0); 
\draw[thick, purple]  (3, 0) --  (2,-1) -- (1,0)  -- (2,1) -- (3,0); 
\draw[thick, purple]  (0, -3) --  (-1,-2) -- (0, -1)  -- (1, -2) -- (0, -3); 
\draw[thick, purple]  (0, 3) --  (-1, 2) -- (0,1)  -- (1,2) -- (0,3); 

\draw (-4,4) node {$L_c(W)$};

\end{scope}
%
%
\begin{scope}[xshift = 4cm, yshift = 0cm, scale = 0.6]

\draw[->, thick] (0,-6.5) -- (0,6.5);
\draw[->, thick] (-6.5,0) -- (6.5,0);
\draw (-0.5,6.1) node {$\zeta$};
\draw (6.0,-0.5) node{$\xi$};
\draw (2.0, -0.5) node{\small $1$};
\draw (2.0,-0.2) -- (2.0, 0);


\draw[violet, fill=violet] (-2,0) circle (0.1cm);
\draw[violet, fill=violet] (2,0) circle (0.1cm);
\draw[violet, fill=violet] (0,-2) circle (0.1cm);
\draw[violet, fill=violet] (0, 2) circle (0.1cm);
\draw[violet, ultra thick] (0,2) -- (0,-2);
\draw[violet, ultra thick] (2,0) -- (-2,0);

\draw[thick, orange]  (-5, 0) --  (0,-5) -- (5,0)  -- (0,5) -- (-5,0); 
\draw[thick, yellow]  (-6, 0) --  (0,-6) -- (6,0)  -- (0,6) -- (-6,0);
\draw[thick, red]  (-4, 0) --  (0,-4) -- (4,0)  -- (0,4) -- (-4,0); 

\draw[thick, purple]  (-3, 0) --  (-2,1) -- (-1,1) -- (-1,2) -- (0, 3) -- (1,2) -- (1,1) -- (2,1) -- (3,0) -- (2,-1) -- (1,-1) -- (1,-2) -- (0,-3)  -- (-1,-2) -- (-1,-1) -- (-2,-1) -- (-3,0); 

\draw (-4,4) node {$L_c(W)^{\rm sc}$};

\end{scope}

\begin{scope}[xshift = -4cm, yshift = -8cm, scale = 0.6]

\draw[->, thick] (0,-5.5) -- (0,5.5);
\draw[->, thick] (-6.5,0) -- (6.5,0);
\draw (-0.5,5.1) node {$\zeta$};
\draw (6.0,-0.5) node{$\xi$};
\draw (1.8, -0.5) node{\small $1$};
\draw (2.0,-0.2) -- (2.0, 0);


\draw[thick, red]  (-2, 2) --  (2,2) -- (2,-2)  -- (-2,-2) -- (-2,2); 
\draw[thick, orange]  (-2.5, 2.5) --  (2.5,2.5) -- (2.5,-2.5)  -- (-2.5,-2.5) -- (-2.5,2.5); 
\draw[thick, yellow]  (-3, 3) --  (3,3) -- (3,-3)  -- (-3,-3) -- (-3,3); 

\draw (-4,4) node {$L_c(W)^\wedge = L_c(\widehat W)$};

\end{scope}

\begin{scope}[xshift = 4cm, yshift = -8 cm, scale = 0.6]

\draw[->, thick] (0,-5.5) -- (0,5.5);
\draw[->, thick] (-6.5,0) -- (6.5,0);
\draw (-0.5,5.1) node {$\zeta$};
\draw (6.0,-0.5) node{$\xi$};
\draw (1.8, -0.5) node{\small $1$};
\draw (2.0,-0.2) -- (2.0, 0);

\draw[thick, purple]  (-1, 1) --  (1,1) -- (1,-1)  -- (-1,-1) -- (-1,1); 
\draw[thick, red]  (-2, 2) --  (2,2) -- (2,-2)  -- (-2,-2) -- (-2,2); 
\draw[thick, orange]  (-2.5, 2.5) --  (2.5,2.5) -- (2.5,-2.5)  -- (-2.5,-2.5) -- (-2.5,2.5); 
\draw[thick, yellow]  (-3, 3) --  (3,3) -- (3,-3)  -- (-3,-3) -- (-3,3); 
\draw[violet, fill=violet] (0,0) circle (0.1cm);

\draw (-4,4) node {$(L_c(W)^{\rm sc})^\wedge$};

\end{scope}
\end{tikzpicture}
\caption{Illustration of the (sub)level sets of $W$ for $p=2$ and their separate convexifications and diagonalizations for the levels $c=0$ (violet), $c =\tfrac{1}{4}$ (purple), $c=1$ (red), $c=\tfrac{5}{4}$ (orange), $c=\tfrac{3}{2}$ (yellow).}
\end{figure}
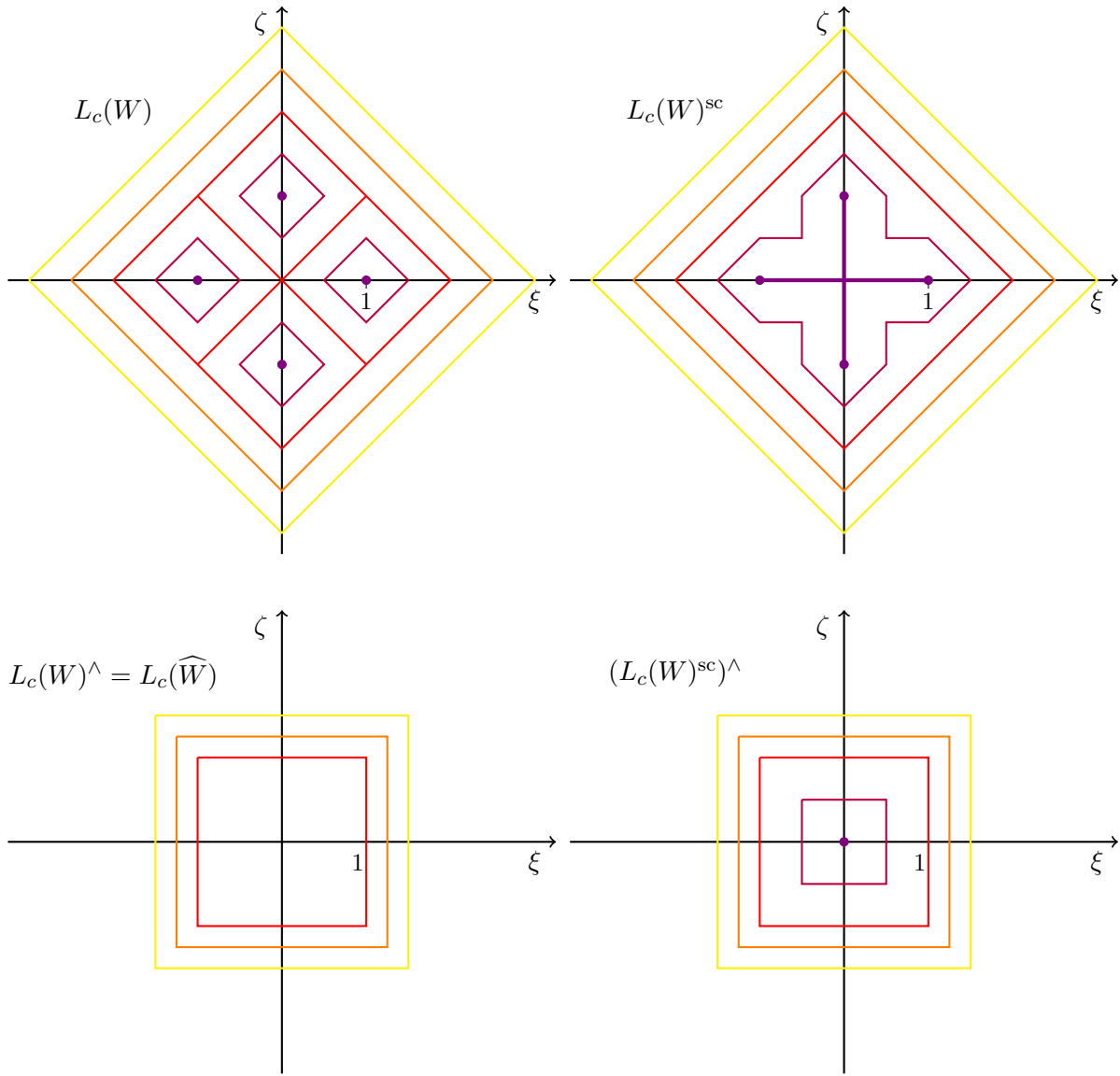

After diagonalization,~\eqref{sublevels} and~\eqref{sublevels2} turn into 
\begin{align}\label{Wex}
L_c(W)^\wedge = \begin{cases} [-c^{1/p},c^{1/p} ]^2 & \text{for $c\geq 1$,}\\
\emptyset & \text{for $c<1$,}
\end{cases}
\qquad \text{and}\qquad 
( L_c(W)^{\rm sc})^\wedge = \begin{cases} [-c^{1/p},c^{1/p} ]^2 & \text{for $c\geq 0$,}\\
\emptyset & \text{for $c<0$,} 
\end{cases}
\end{align}
for $c\in \R$, see Figure~\ref{fig:levelsets}. Observe in particular that $L_0(W)^{\wedge}=\widehat K = \emptyset$. In view of~\eqref{identity_LcW} and~\eqref{Wex}, one can deduce an explicit expression for the diagonalization of $W$, that is, 
\begin{align*}
\widehat W(\xi, \zeta) = \dist_\infty^p((\xi, \zeta), [-1,1]^2) +1\quad 
\text{for $(\xi, \zeta)\in \R^m$.}
\end{align*} 
On the other hand, by \eqref{sublevels2}, \eqref{Wex} and \eqref{identity_LcW}$,  L_0(\widehat{W^{\rm sc}})=L_0(W^{\rm sc})^\wedge\supset\{(0,0)\}$.

Summing up, the previous calculations show that
\begin{align}\label{ref1a}
\min W =0<1= \min \widehat W\quad\text{ and }\quad \min \widehat{W^{\rm sc}} = \min W^{\rm sc} = \min W=0.
\end{align} 

Thus, according to Corollary~\ref{cor:neq}, separate convexification of the double integrand $W$ fails to give a representation for $I_W^{\rm rlx}$. 
The next result provides even more, namely that the relaxation of $I_W$ is not of double-integral form at all.

\begin{Proposition}\label{prop:counterexample}
Let $W:\R\times \R\to \R$ as in~\eqref{Wexample}.
There exists no symmetric, lower semicontinuous double integrand $G: \R\times \R\to \R$ with $p$-growth such that 
$I_W^{\rm rlx}= I_G$. 
\end{Proposition}

\begin{proof}
We argue by contradiction, and suppose therefore that 
\begin{align}\label{est45}
0\leq I_{\rm W^{\rm sc}}\leq I_W^{\rm rlx} = I_G\leq I_W
\end{align} for some $G:\R\times \R\to\R$ as in the statement. 

A comparison of the sublevel sets of $W$ and $W^{\rm co}$ in~\eqref{sublevels} and \eqref{lcWco} shows that 
\begin{align*}
L_c(W)\setminus (-1,1)^2 = L_c(W^{\rm co}) \setminus (-1,1)^2
\end{align*}
which entails that $W$ coincides with $W^{\rm co}$ outside of $(-1,1)^2$. In view of~\eqref{2.3},
\begin{align*}
\text{$W= W^{\rm slc}=W^{\rm sc}= W^{\rm co}$ on $[(-1,1)^2]^c$. }
\end{align*}
 We invoke Proposition~\ref{intrep} to infer that $G=W$ on the complement of $(-1,1)^2$, and therefore, $G=0$ on $K$. 
Moreover, $G(0,0)\geq W^{\rm sc}(0,0)\geq 0$ by Lemma \ref{lem:FG}. 
Consequently, the $L^p$-weak lower semicontinuity of $I_G=I_W^{\rm rlx}$, which yields that $G$ is separately convex (see e.g.~\cite[Theorem~1.1]{BeP06}), leads us to conclude that $G$ vanishes on $\{0\}\times [-1,1]\cup [-1,1]\times \{0\}= K^{\rm sc}$; in particular, $G(0,0)=0$. 

This proves that $\min_{u\in L^p(\Omega)} I_G(u)= I_G(0) = 0$, cf.~\eqref{est45}. As $I_G$ coincides with the relaxation of $I_W$ by assumption, one has that
\begin{align}\label{856}
\inf_{u\in L^p(\Omega)} I_W(u) = \min_{u\in L^p(\Omega)} I_W^{\rm rlx}(u) = \min_{u\in L^p(\Omega)} I_G(u) = 0.
\end{align}

However, by~Proposition~\ref{prop:geq} in combination with~\eqref{ref1a}, 
$\inf_{u\in L^p(\Omega)}I_W(u) > |\Omega|^2\min W =0$, which contradicts~\eqref{856} and finishes the proof.
\end{proof}

\begin{Remark} 
Alternatively, there is also a direct and self-contained argument for the last step in the previous proof, meaning, one that  does not make use of Proposition~\ref{prop:geq}. Following along the lines of \cite[Section~3]{Ped16}, we argue by contradiction and let $\inf_{u\in L^p(\Omega)} I_W(u)=0$, so that
\begin{align}\label{IWto0}
I_W(u_j) = \int_\Omega\int_\Omega W(v_{u_j})\, dx \, dy \to 0 \quad \text{as $j\to \infty$}
\end{align} for some sequence $(u_j)_j\subset L^p(\Omega)$. 
Due to $W\geq 0$, we know that $(v_{u_j})_j$ needs to concentrate around $L_0(W)=K$, and since $K$ is finite, it has to concentrate partially around at least one point, without loss of generality $(1,0)$. Then, there exists a set $\omega\subset \Omega$ of positive measure where $(u_j)_j$ concentrates around $1$, which again entails that $(v_{u_j})_j$ concentrates around $(1,1)$ on $\omega\times \omega$, and thus,  
\begin{align*}
\int_\omega\int_\omega W(v_{u_j})\, dx\,dy \to |\omega|^2\, W(1,1)  \quad \text{as $j\to \infty$.}
\end{align*} 
Observing that the limit is strictly positive because $(1,1)\notin K=L_0(W)$ and $\omega$ has non-vanishing $\mathcal L^{n}$-measure, produces a contradiction with~\eqref{IWto0}. 
\end{Remark}

Not only homogeneous double integrals fail to provide an explicit representation for the $L^p$-weak lower semicontinuous envelope of $I_W$, allowing for inhomogeneous double integrands does not help in obtaining correct relaxation formulas either. 

\begin{Remark} In generalization of Proposition~\ref{prop:counterexample}, one can show that it is not possible to express $I^{\rm rlx}_W$ with $W$ as in~\eqref{Wexample} in terms of
\begin{align*}
L^p(\Omega)\ni u\mapsto \int_{\Omega}\int_{\Omega} G(x,y, u(x), u(y)) \,dx\,dy,
\end{align*}
where $G:\Omega\times \Omega\times\R\times \R\to \R$ is a symmetric and normal function with $p$-growth, i.e., $G$ satisfies
\begin{itemize}
\item[$(i)$] $G(x,y, \xi, \zeta) = G(y, x, \xi, \zeta)$ and $G(x,y, \xi, \zeta) = G(x,y, \zeta, \xi)$ for all $x,y\in \Omega$ and $\xi, \zeta\in \R$;
\item[$(ii)$] $G(x,y, \cdot, \cdot)$ is lower semicontinuous for a.e.~$(x,y)\in \Omega\times \Omega$ and  $G(\cdot, \cdot, \xi, \zeta)$ is measurable for all $(\xi, \zeta)\in \R\times \R$; 
\item[$(iii)$] $|G(x,y, \xi, \zeta)|\leq C (a(x,y)+ |\xi|^p + |\zeta|^p)$ for all $(x,y)\in \Omega\times\Omega$ and $(\xi, \zeta)\in \R\times \R$ with a constant $C>0$ and $a\in L^1(\Omega\times \Omega)$.
\end{itemize}

To see this, it suffices to substitute $G$ in the proof of Proposition~\ref{prop:counterexample} by 
\begin{align*}
\overline{G}(\xi, \zeta) := \int_{\Omega}\int_{\Omega} G(x,y, \xi, \zeta)\, dx\, dy \quad \text{for $(\xi, \zeta)\in \R\times \R$,}
\end{align*} 
and to use~\cite[Theorem~2.5]{Ped16} in place of~\cite[Theorem~1.1]{BeP06}.
\end{Remark}

\subsection{Second counterexample}\label{subsec:counterexample2}
The double integrand for our second example is qualitatively different from the first, in the sense that its minimum does not change under diagonalization. 

\begin{Proposition}\label{prop:counterexample2}
Let $K= \partial B_1^{1}(0,0)= \{(\xi, \zeta)\in \R\times \R: |\xi| + |\zeta| = 1\}$ and let $W:\R\times \R\to \R$ be given by 
\begin{align}\label{densityW2}
W(\xi, \zeta) = {\rm dist}^p_1((\xi, \zeta), K) \quad \text{ for $(\xi, \zeta)\in \R\times  \R$.}
\end{align} 
There exists no symmetric, lower semicontinuous double integrand $G: \R\times \R\to \R$ with $p$-growth such that 
$I_W^{\rm rlx}= I_G$.
\end{Proposition}

\begin{proof} Arguing by contradiction, we suppose that $I_W^{\rm rlx}=I_G$ with $G$ as in the statement, and split the proof into two steps. First, Step~1 shows that necessarily $G=W^{\rm sc}$, and then, we conclude in Step~2 that $I_{W}^{\rm rlx}\neq  I_{W^{\rm sc}}$, which yields the desired contradiction.

\textit{Step~1: $G=W^{\rm sc}$.} 
Since $K^{\rm sc}=K^{\rm co}= B_1^{1}(0,0)$, it follows from Lemma~\ref{lem:Wsc=Wco} that $W^{\rm sc}=W^{\rm co}$. With $W^{\rm co} = {\rm dist}^p_1(\cdot, K^{\rm co})$ by~\eqref{Wco_distance}, we see that $W$ and $W^{\rm sc}$ differ only in the interior of $K^{\rm co} = B_1^{1}(0,0)$,  which we will denote by $B$ in the following; hence, 
\begin{align}\label{Bc}
W^{\rm sc} = W \quad \text{on $B^c$.}  
\end{align}
A comparison between $G$, $W$ and $W^{\rm sc}$ along the diagonal yields
\begin{align*}
W^{\rm sc}(\xi, \xi) = G(\xi, \xi)=W(\xi, \xi)\quad \text{ for $\xi\in \R$ with $|\xi|\geq \tfrac{1}{2}$,}
\end{align*}
according to Lemma~\ref{lem:FG}; note that like in Section~\ref{section:ex1},~\eqref{est45} has to hold here as well.  
In view of Proposition~\ref{intrep}, $W^{\rm sc}\leq G\leq W$ on $[(-\frac{1}{2}, \frac{1}{2})^2]^c$, we obtain together with~\eqref{Bc} that
\begin{align}\label{G=Wsc1}
G=W^{\rm sc}\quad \text{ on 
 $B^c$.} 
\end{align}

Next, we prove that 
\begin{align}\label{G=Wsc2}
G=W^{\rm sc}\quad \text{ on $[-\tfrac{1}{2}, \tfrac{1}{2}]^2$.}
\end{align}
The argument is based on the observation that $\widehat K = \{-\frac{1}{2}, \frac{1}{2}\}^2$, and therefore $\widehat K^{\rm sc}=\widehat K^{\rm co}=[-\frac{1}{2}, \frac{1}{2}]^2$. 

For $(\xi, \zeta)\in [-\frac{1}{2}, \frac{1}{2}]^2$, let
 $v=\xi\mathbbm{1}_{\Omega_\xi} + \zeta \mathbbm{1}_{\Omega\setminus \Omega_\xi}$ with a measurable set $\Omega_\xi\subset \Omega$ of measure $\lambda |\Omega|$ with $\lambda\in (0,1)$. By Lemma~\ref{lem:oscillations}, one can find a sequence $(v_j)_j\subset S^\infty(\Omega)$ of simple functions that oscillate suitably between the values $\pm\frac{1}{2}$ such that $v_j\weaklystar v$ in $L^\infty(\Omega)$. 
As 
\begin{align*}
0 \leq I_{W^{\rm sc}}(v) = I_{G}(v) = I_W^{\rm rlx}(v) \leq \limsup_{j\to \infty} I_W(v_j) = 0,
\end{align*} 
owing to $W=0$ on $\widehat K\subset K$, it follows that 
\begin{align*}
\lambda^2 G(\xi, \xi) + (1-\lambda)^2 G(\zeta,\zeta) + 2 \lambda (1-\lambda)G(\xi, \zeta)=0\quad\text{ for all $\lambda\in (0,1)$.}
\end{align*} 
Letting $\lambda\to 0$ and $\lambda\to 1$ yields first that $G(\xi, \xi)=G(\zeta, \zeta)=0$, and eventually, also $G(\xi, \zeta)=0$. 
This finishes the proof of~\eqref{G=Wsc2}. 

We can now infer from~\eqref{G=Wsc1} and~\eqref{G=Wsc2} that $G$ coincides with $W^{\rm sc}$ everywhere on the diagonal, and since $I_{W^{\rm sc}}\leq I_W^{\rm rlx} =I_G$, Lemma~\ref{lem:FG} implies that 
$G\geq W^{\rm sc}\geq 0$ on $\R\times \R$. 
Since $G$ is non-negative with $G=0$ on $\partial B_1^1(0,0)=\partial B$ by~\eqref{G=Wsc1} and $G$ has to be separately convex (see e.g.~\cite[Theorem~1.1]{BeP06}), it follows that
\begin{align}\label{GWsc}
G=0  \quad \text{on $B_1^1(0,0)$.}
\end{align} 
In combination with~\eqref{G=Wsc1}, this shows $G=W^{\rm sc}$.

\textit{Step~2: $I_W^{\rm rlx}\neq I_{W^{\rm sc}}$.} 
Considering the simple function $v=\mathbbm{1}_{\Omega_1}$, where $\Omega_1\subset \Omega$ is a set of positive Lebesgue measure, we aim to show that 
\begin{align*}
 I_{W}^{\rm rlx}(v) \neq I_{W^{\rm sc}}(v).
 \end{align*}
Assume to the contrary that $I_{W}^{\rm rlx}(v) = I_{W^{\rm sc}}(v)$. Then, by the definition of the relaxed functional, there exists a sequence $(u_j)_j\subset L^p(\Omega)$ such that $u_j\weakly v$ in $L^p(\Omega)$ and
\begin{align}\label{est98}
\lim_{j\to \infty} I_W(u_j)& = I_{W^{\rm sc}}(v) =  |\Omega_1|^2 W^{\rm sc}(1,1) + |\Omega\setminus \Omega_1|^2 W^{\rm sc}(0,0) + 2|\Omega_1| |\Omega\setminus \Omega_1|W^{\rm sc}(1,0) =  |\Omega_1|^2;
\end{align} 
here, in the last step we have exploited again that $W^{\rm sc}=W^{\rm co}=\dist^p_1(\cdot, K^{\rm co})$. 
On the other hand, along with the weak convergence of the restrictions of $(u_j)_j$ to $\Omega_1$ and $\Omega\setminus \Omega_1$, i.e.~$u_j|_{\Omega_1}\weakly 1$ in $L^p(\Omega_1)$ and  $u_j|_{\Omega\setminus \Omega_1}\weakly 0$ in $L^p(\Omega\setminus \Omega_1)$, as well as the symmetry and non-negativity of $W$, 
\begin{align}  
\displaystyle\lim_{j\to \infty} I_W(u_j) & = \lim_{j\to \infty}\Bigl(\int_{\Omega_1}\int_{\Omega_1} W(u_j(x), u_j(y)) \, dx\, dy+ \int_{\Omega\setminus\Omega_1}\int_{\Omega\setminus \Omega_1} W(u_j(x), u_j(y)) \, dx\, dy  \nonumber\\ &\qquad +  2\int_{\Omega_1}\int_{\Omega\setminus \Omega_1} W(u_j(x), u_j(y)) \, dx\, dy\Bigr) \nonumber \\ 
& \geq |\Omega_1|^2W^{\rm sc}(1,1) + |\Omega\setminus\Omega_1|^2W^{\rm sc}(0,0)  +2\, \liminf_{j\to \infty} \int_{\Omega_1}\int_{\Omega\setminus \Omega_1} W(u_j(x), u_j(y)) \, dx\, dy \label{est74}\\ 
& = |\Omega_1|^2  +2\, \liminf_{j\to \infty} \int_{\Omega_1}\int_{\Omega\setminus \Omega_1} W(u_j(x), u_j(y)) \, dx\, dy\geq |\Omega_1|^2. \nonumber 
\end{align}
Combining~\eqref{est74} with~\eqref{est98} turns all equalities in~\eqref{est74} into equalities. Hence, after passing to a suitable (not relabeled) subsequence, 
\begin{align}\label{conv_11}
\lim_{j\to \infty} \int_{\Omega_1}\int_{\Omega_1} W(u_j(x), u_j(y)) \, dx\, dy = |\Omega_1|^2,
\end{align}
\begin{align}\label{conv_22}
\lim_{j\to \infty} \int_{\Omega\setminus\Omega_1}\int_{\Omega\setminus \Omega_1} W(u_j(x), u_j(y)) \, dx\, dy = 0,
\end{align}
and
\begin{align}\label{conv_33}
\lim_{j\to \infty} \int_{\Omega_1}\int_{\Omega\setminus \Omega_1} W(u_j(x), u_j(y)) \, dx\, dy = 0.
\end{align}

The identity~\eqref{conv_22} shows that the sequence $(u_j|_{\Omega\setminus \Omega_1})_j\subset L^p(\Omega\setminus \Omega_1)$ concentrates around $\pm\frac{1}{2}$, while~\eqref{conv_11} indicates that this is not the case for $(u_j|_{\Omega_1})_j\subset L^p(\Omega_1)$.  Now, let $\omega\subset \Omega\setminus \Omega_1$ and $\omega_1\subset \Omega_1$ be sets of positive $\mathcal L^n$-measure and $\varepsilon>0$ such that $(u_j|_{\omega})_j$ concentrates around $\frac{1}{2}$, and $(u_j|_{\omega_1})_j$ concentrates on the complement of $(\frac{1}{2}-\varepsilon, \frac{1}{2} +\varepsilon)\cup (-\frac{1}{2}-\varepsilon, -\frac{1}{2} +\varepsilon)$. 
 Such sets exist without loss of generality, otherwise replace $\frac{1}{2}$ by $-\frac{1}{2}$ for the set of concentrations of $(u_j|_\omega)_j$. With $\nu\in L_w^\infty(\omega_1;\mathcal Pr(\R))$ the Young measure generated by a suitable (non-relabeled) subsequence of $(u_j|_{\omega_1})_j$, it follows that 
\begin{align}\label{est77}
\lim_{j\to \infty} \int_{\omega}\int_{\omega_1} W(u_j(x), u_j(y)) \, dx\, dy =  |\omega| \int_{\omega_1} \int_{\R} W( \xi, \tfrac{1}{2}) \, d\nu_x(\xi)\, dx >0,
\end{align}
which contradicts~\eqref{conv_33}. 
The estimate in~\eqref{est77} makes use of the fact that by the choice of $\omega_1$, $\pm\frac{1}{2}\notin {\rm supp\,} \nu_x$ for a.e.~$x\in \omega_1$, along with the observation that $W(\xi, \frac{1}{2})=0$ if and only if $\xi\in \{-\frac{1}{2}, \frac{1}{2}\}$. 
\end{proof}

\begin{Remark}\label{rem:notsufficient}
We notice that condition~\eqref{ness} from Corollary~\ref{prop:ness}, which is necessary for structure-preserving relaxation of double integrals via separate convexification in the scalar setting, is in general not sufficient. 

Indeed, for the double integrand $W$ introduced in~\eqref{densityW2}, one has that $L_0(W)=K$ with $\widehat K=\{-\frac{1}{2}, \frac{1}{2}\}^2$, $K^{\rm sc} =K^{\rm co}= B_1^1(0,0)$ and $W^{\rm sc}=W^{\rm co}$; therefore, 
\begin{align*}
(L_0(W)^\wedge)^{\rm sc} =\widehat K^{\rm sc} = [-\tfrac{1}{2}, \tfrac{1}{2}]^2 = B_1^1(0,0)^\wedge = \widehat{K^{\rm co}} = L_0(W^{\rm co})^\wedge= L_0(W^{\rm sc})^\wedge,
\end{align*}
which is~\eqref{ness}. 
\end{Remark}

With Propositions~\ref{prop:counterexample} and~\ref{prop:counterexample2} at hand, it is not hard to generate counterexamples 
 also in the vectorial setting. 

\begin{Remark}\label{rem:generalization}
Let  $m\in \mathbb N$ with $m>1$ and $K$ as in Proposition~\ref{prop:counterexample} or~\ref{prop:counterexample2}. 
For $\xi, \zeta \in \R^m$, let $\xi', \zeta'\in \R^{m-1}$ be the vectors of the last $m-1$ components of $\xi$ and $\zeta$, respectively, so that $\xi = (\xi_1, \xi') \in \R^m$ and $\zeta = (\zeta_1, \zeta')\in \R^m$.

We define
\begin{align*}
W(\xi, \zeta) = {\rm dist}_1^p((\xi_1, \zeta_1), K) + |(\xi', \zeta')|^p =: W_1(\xi_1, \zeta_1) + W'(\xi', \zeta')
\end{align*}
for $(\xi, \zeta)\in \R^m\times \R^m$. 
Due to the decoupling of the first component from the last ones, it is straightforward to check that
\begin{align*}
I_W^{\rm rlx}(u)= I_{W_1}^{\rm rlx}(u_1) + I_{W'}(u') \qquad \text{for $u=:(u_1, u')\in L^p(\Omega;\R^m)$.}
\end{align*}  
Therefore, since $I_{W_1}^{\rm rlx}:L^p(\Omega)\to \R$ fails to be a double integral, so does $I_W^{\rm rlx}:L^p(\Omega;\R^m)\to \R$. 
\end{Remark} 
%
%

\section{Examples of structure-preserving relaxation}\label{sec:examples_relaxation}
In this last section, we provide examples of non-trivial relaxation where the double-integral structure is preserved and the integrands result from taking the separately convex envelope. 

\subsection{Integrands of distance type.}
Let $K=A\times A$ with a compact set $A\subset \R^m$ and $1\leq q\leq \infty$. 
We consider functions $W:\R^m\times \R^m\to \R$ defined via 
\begin{align}\label{W3}
W(\xi, \zeta)=\dist^p_q((\xi, \zeta), K)  = \dist^p_q((\xi, \zeta), A\times A)\quad \text{for $(\xi, \zeta)\in \R^m\times \R^m$,}
\end{align}
cf.~\eqref{dist}. 

As $K^{\rm sc} = K^{\rm co} = A^{\rm co}\times A^{\rm co}$ according to~\eqref{AcoxAco}, Lemma~\ref{lem:Wsc=Wco} implies that 
the separately convex envelope of $W$ is identical with the convex one, that is, 
\begin{align}\label{WscExample2}
W^{\rm sc}(\xi,\zeta) & = W^{\rm co}(\xi,\zeta)  
= \dist^p_q((\xi, \zeta), A^{\rm co}\times A^{\rm co}) \\ 
& =\begin{cases} 
\bigl(\dist^q(\xi, A^{\rm co}) + \dist^q(\zeta, A^{\rm co})\bigr)^{\frac{p}{q}} &\text{if $1\leq q<\infty$,}\\
\max\{\dist(\xi, A^{\rm co}), \dist(\zeta, A^{\rm co})\}^p & \text{if $q=\infty$,}
\end{cases}\nonumber
\end{align}
for $(\xi, \zeta)\in \R^m\times \R^m$; notice that the function $\dist(\cdot, A^{\rm co})$ on $\R^m$ denotes the Euclidean distance from $A^{\rm co}$. 

Under the additional hypothesis on $A$ that 
\begin{align}\label{hypoA}
\dist(\xi, A^{\rm co}) = \dist(\xi, A) \qquad \text{for all $\xi \notin A^{\rm co}$,}
\end{align}  
one can express the (separate) convexification of $W$ in the following way:  With any $\alpha, \beta\in A$,  
\begin{align}\label{WscExample}
W^{\rm sc}(\xi, \zeta) = \begin{cases} 0 & \text{if $\xi\in A^{\rm co}$ and $\zeta\in A^{\rm co}$,}\\
W(\xi,\zeta) & \text{if $\xi \notin A^{\rm co}$ and $\zeta\notin A^{\rm co}$,}\\
W(\alpha, \zeta)& \text{if $\xi\in A^{\rm co}$ and $\zeta\notin A^{\rm co}$,}\\
W(\xi, \beta) & \text{if $\xi\notin A^{\rm co}$ and $\zeta\in A^{\rm co}$,}
\end{cases} 
\end{align}
for $(\xi, \zeta)\in \R^m\times \R^m$. 
Notice that the condition~\eqref{hypoA} is always fulfilled for $A$ if $m=1$;
a sufficient condition for~\eqref{hypoA} in the vectorial case $m>1$ is for instance that $A\subset \R^m$ satisfies $\partial A^{\rm co} = \partial A$. 

We have the following characterization result. 

\begin{Proposition}\label{prop:fourwell}
Let $W$ be as in~\eqref{W3} with $A$ satisfying~\eqref{hypoA}. Then, $I_W^{\rm rlx} = I_{W^{\rm sc}}$.  
\end{Proposition}

\begin{proof}
Since $W^{\rm sc}$ is separately convex, and hence $I_{W^{\rm sc}}$ $L^p$-weakly lower semicontinuous  (see~e.g.~\cite[Theorem~1.1]{Mun09},~\cite[Theorem~2.6]{Ped16}), it is clear that $I^{\rm rlx}_W\geq I_{W^{\rm sc}}$. To prove the reverse inequality, let $u\in L^p(\Omega;\R^m)$, which we approximate by a sequence of simple functions $(u_k)_k$ such that 
$u_k\to u$ in $L^p(\Omega;\R^m)$. Under consideration of the continuity and $p$-growth of $W^{\rm sc}$, the Vitali-Lebesgue convergence theorem yields that $I_{W^{\rm sc}}(u) = \lim_{k\to \infty} I_{W^{\rm sc}}(u_k)$. 
It remains to find for any simple function 
\begin{align}\label{simple_v}
v=\sum_{i=1}^{N} \xi^{(i)}\mathbbm{1}_{\Omega^{(i)}}
\end{align}
with $\xi^{(i)}\in \mathbb{R}^m$ and $\{\Omega^{(i)}\}_{i=1, \ldots,N}$ a decomposition of $\Omega$ into measurable subsets, a sequence $(v_j)_j\subset L^p(\Omega;\R^m)$ such that $v_j\weakly v$ in $L^p(\Omega;\R^m)$,  
and 
\begin{align}\label{est1}
\limsup_{j\to \infty} I_W(v_j) \leq I_{W^{\rm sc}}(v); 
\end{align} 
the claim follows then from a diagonalization argument. 

In the following, we take $v$ as in~\eqref{simple_v} and detail the construction of $(v_j)_j$ with the desired properties. If $\xi^{(i)}\in A^{\rm co}$, 
let $(v^{(i)}_j)_j\subset L^p(\Omega^{(i)};\R^m)$ be a sequence that converges weakly to $v$ in $L^p(\Omega^{(i)};\R^m)$ and takes values only in $A$, meaning,
\begin{align}\label{123}
v^{(i)}_j\in A \qquad \text{a.e.~in $\Omega^{(i)}$ for all $j\in \mathbb N$,}
\end{align} 
cf.~Lemma~\ref{lem:oscillations}. 
If $\xi^{(i)}\notin A^{\rm co}$, let $v^{(i)}_j$ for any $j\in \mathbb N$ be the constant function on $\Omega^{(i)}$ with value $\xi^{(i)}$. With these definitions, consider the sequence $(v_j)_j\subset L^p(\Omega;\R^m)$ given by 
\begin{align*}
v_j:=\sum_{i=1}^N v_j^{(i)}\mathbbm{1}_{\Omega^{(i)}} \qquad \text{for $j\in \mathbb N$.}
\end{align*}

By construction, $v_j \weakly v$ in $L^p(\Omega;\R^m)$, and 
\begin{align}\label{eqlast}
I_W (v_j) & = \int_\Omega\int_\Omega W(v_j(x), v_j(y))\, dx\, dy = \sum_{i, l=1}^{N} \int_{\Omega^{(i)}}\int_{\Omega^{(l)}} W(v_j^{(i)}(x), v_j^{(l)}(y))\, dx\,dy  \nonumber \\ & 
= \sum_{i, l=1}^{N} |\Omega^{(i)}|\,|\Omega^{(l)}| W^{\rm sc}(\xi^{(i)}, \xi^{(l)})  =  I_{W^{\rm sc}}(v) 
\end{align}
for all $j\in \mathbb{N}$, which implies~\eqref{est1} and concludes the proof; the third identity in~\eqref{eqlast} follows from the observation that for any $j\in \mathbb{N}$ and $i, l\in \{1, \ldots, N\}$, 
\begin{align}\label{ref}
W(v_j^{(i)}(x), v_j^{(l)}(y)) = 
W^{\rm sc}(\xi^{(i)}, \xi^{(l)})
 \qquad \text{for a.e.~$(x,y)\in \Omega^{(i)}\times \Omega^{(l)}$.}
\end{align}
 To see the latter, we distinguish three different cases. 
If $\xi^{(i)}, \xi^{(l)}\notin A^{\rm co}$, the functions $v_j^{(i)}$ and $v_j^{(l)}$ are constant, and $W^{\rm sc}(\xi^{(i)}, \xi^{(l)}) =  W(\xi^{(i)}, \xi^{(l)})$ by~\eqref{WscExample}. For $\xi^{(i)}, \xi^{(l)}\in A^{\rm co}$, both expressions in~\eqref{ref} are zero (almost everywhere) according to~\eqref{123} and~\eqref{WscExample}.
In the case $\xi^{(i)}\in A^{\rm co}$ and $\xi^{(l)}\notin A^{\rm co}$, we invoke again~\eqref{WscExample} to obtain that $W^{\rm sc}(\xi^{(i)},\xi^{(l)}) = W(\alpha, \xi^{(l)})$ for any $\alpha\in A$; hence,~\eqref{ref} holds in light of~\eqref{123}.
For $\xi^{(i)}\notin A^{\rm co}$ and $\xi^{(l)}\in A^{\rm co}$, the reasoning is analogous. 
\end{proof}

\subsection{Indicator functionals} 
For a symmetric, diagonal and compact set $K\subset \R^m\times \R^m$, we define the associated indicator functional $J_K$ via 
\begin{align}\label{indicator_intro}
L^p(\Omega;\R^m)\ni u\mapsto J_K(u):= I_{\chi_K}(u) = \int_\Omega\int_\Omega \chi_K(u(x), u(y))\,dx \,dy,
\end{align} 
where $\chi_K$ is the indicator function of $K$, 
  i.e., $\chi_K(\xi, \zeta)=0$ if $(\xi, \zeta)\in K$ and $\chi_K=\infty$ otherwise in $\R^m\times \R^m$. Note that $J_K$ can be expressed in terms of nonlocal inclusions as
\begin{align*}
J_K(u)=\begin{cases} 0 & \text{if $u\in \mathcal A_K$,}\\
\infty & \text{otherwise,}
\end{cases}
\end{align*}
for $u\in L^p(\Omega;\R^m)$; recall the definition of $\mathcal A_K$ in Section~\ref{sec:nonlocalinclusions}.

According to~\cite[Corollary~6.2]{KrZ19} and Theorem \ref{theo:Young}, the Young measure relaxation of $J_K$ is
	\begin{align*}
	J_K^{\mathcal Y}(\nu) &=\begin{cases} 0 &\text{
		${\rm supp\,}\nu\otimes \nu\subset \widehat K$ a.e.~in $\Omega\times \Omega$,} \\ \infty & \text{otherwise,} \end{cases}\\
	&= \begin{cases} 0 &\text{ 
		${\rm supp\,}\nu\otimes \nu\subset P$ a.e.~in $\Omega\times \Omega$ with $P\in \mathcal P_K$,} \\ \infty & \text{otherwise, } \end{cases}\\
	& =\begin{cases} 0 &\text{
		${\rm supp\,}\nu \subset A$ a.e.~in $\Omega$ with $A\times A\in \mathcal P_K$,} \\ \infty & \text{otherwise,} \end{cases}
	\end{align*}
	for $\nu\in L^\infty_w(\Omega;\mathcal Pr(\R^m))$, and~\eqref{Acal_Kinfty} provides
 a representation formula for the relaxation of $J_K$ with respect to the $L^p$-weak topology; precisely, for $u\in L^p(\Omega;\R^m)$,
  \begin{align}\label{relax_indi}
J_K^{\rm rlx}(u)& : = \inf \{ \liminf_{j\to \infty} J_K(u_j): 
 u_j\weakly u \text{ in $L^p(\Omega;\R^m)$}\} \\ &= \begin{cases} 0 & \text{if $u\in \mathcal B_{K, p}^\infty=\mathcal A_K^\infty$,} \nonumber \\
\infty & \text{otherwise,}
\end{cases}\ \\ & = \begin{cases} 0 & \text{if $u\in A^{\rm co}$ a.e.~in $\Omega$ with $A\times A\in \mathcal P_K$,}\\
\infty & \text{otherwise,}
\end{cases}  \\ & = \int_\Omega\int_\Omega \,\chi_{\bigl[\bigcup_{A\times A \in \mathcal P_K} A^{\rm co}\times A^{\rm co}\bigr]}(u(x), u(y)) \, dx\, dy = J_{K^{\rm rlx}},\nonumber
\end{align}
with $K^{\rm rlx}:=\bigcup_{A\times A \in \mathcal P_K} A^{\rm co}\times A^{\rm co}$; for the second identity, we invoke Theorem~\ref{prop:approx_inclusion} in combination with Remark~\ref{rem:BKinftyp}. 

It is generally not true that $K^{\rm rlx}$ coincides with the separately convex hull of $K$ (see~\cite[Example~4.6\,b)]{KrZ19}); yet, under the additional assumption that 
\begin{align*}
\widehat{K^{\rm sc}}
= \bigcup_{(\alpha, \beta)\in K} \ [\alpha, \beta]\times [\alpha, \beta],
\end{align*} 
which is for instance satisfied for $m=1$ (see~\cite[Lemma~4.7]{KrZ19}), 
it was shown in~\cite[Corollary~6.1]{KrZ19} that $J_{K}^{\rm rlx}=J_{K^{\rm sc}}$. Whether this identity holds in general, or equivalently, if $K^{\rm rlx} = \widehat {K^{\rm sc}}$ without further assumptions on $K$, 
remains unknown. 

In conclusion, we have seen that 
the relaxation of indicator functionals of the type~\eqref{indicator_intro} is always structure preserving.
%

\section*{Acknowledgements}
The authors would like to thank Stefan Kr\"omer for inspiring discussions.
CK was partially supported by a Westerdijk Fellowship from Utrecht University, by the Dutch Research Council NWO through the project TOP2.17.012, and by INdAM -GNAMPA through Programma Professori Visitatori a.a.~2018/2019. The hospitality of Dipartimento di Ingegneria Industriale at University of Salerno is gratefully acknowledged.  
Both authors would like to thank the Isaac Newton Institute for Mathematical Sciences, Cambridge, for its support and hospitality during the programme DNM (Design of New Materials), where work on this paper was undertaken. This work was supported by EPSRC grant no EP/R014604/1.


\bibliographystyle{abbrv}
\bibliography{NonlocalSup}

	\end{document}